\documentclass[10pt]{amsart}
\usepackage{amsmath, amssymb, amscd, mathtools, graphicx, amsfonts, enumerate, mathrsfs, xcolor, yhmath, xcolor}
\newtheorem{theorem}{Theorem}[section]
\newtheorem{proposition}[theorem]{Proposition}
\newtheorem{lemma}[theorem]{Lemma}
\newtheorem{corollary}[theorem]{Corollary}

\theoremstyle{definition}

\newtheorem{example}[theorem]{Example}
\newtheorem{remark}[theorem]{Remark}

\numberwithin{equation}{section}

\newcommand{\N}{\mathbb{N}}                        
\newcommand{\I}{\mathcal{I}}
\newcommand{\R}{\mathbb{R}}                        

\newcommand{\Z}{\mathcal{Z}}

\newcommand{\codim}{\mathrm{codim}}
\newcommand{\domain}{\mathrm{dom}}
\newcommand{\indet}{\mathrm{indet}}

\newcommand{\Image}{\mathrm{Im}}

\newcommand{\Arc}[1]{\wideparen{#1}}
\newcommand{\Tld}[1]{\widetilde{#1}}
\newcommand{\graph}[1]{\mathcal{G}_{#1}} 
\newcommand{\Reg}[1]{\mathcal{R}(#1)} 
\newcommand{\Pol}[1]{\mathcal{P}(#1)} 
\newcommand{\Rrat}[1]{R_{b}(#1)} 
\newcommand{\Proj}[1]{\mathbb{P}^1(#1)}
\DeclarePairedDelimiter\Ang{\langle}{\rangle}
\newcommand{\Puis}[1]{R \langle \langle #1 \rangle \rangle}

\begin{document}

\title{The Geometry of Locally Bounded Rational Functions}

\author{Victor Delage}
\address{Victor Delage, IRMAR (UMR 6625), Universit\'{e} de Rennes - 1, Campus de Beaulieu, 35042 Rennes Cedex, France.}
\email{victor.delage.maths@proton.me}

\author{Goulwen Fichou}
\address{Goulwen Fichou, IRMAR (UMR 6625), Universit\'{e} de Rennes - 1, Campus de Beaulieu, 35042 Rennes Cedex, France.}
\email{goulwen.fichou@univ-rennes.fr}

\author{Aftab Patel}
\address{Aftab Patel, IRMAR (UMR 6625), Universit\'{e} de Rennes - 1, Campus de Beaulieu, 35042 Rennes Cedex, France.}
\email{aftab-yusuf.patel@univ-rennes.fr}

\subjclass[2010]{14P99, 14E05, 14F17, 26C15}
\keywords{regulous functions, rational continuous functions, {\L}ojasiewicz's inequality, arc-spaces, locally bounded function, real algebraic geometry}
\begin{abstract}	
    This paper develops the geometry of locally bounded rational functions on non-singular real
    algebraic varieties. First various basic geometric
    and algebraic results regarding these functions are established in any
    dimension, culminating with a version of {\L}ojasiewicz's inequality.  The
    geometry is further developed for the case of dimension 2, where it can be
    shown that there exist many of the usual correspondences between the
    algebra and geometry of these functions that
    one expects from complex algebraic geometry and from other classes
    of functions in real algebraic geometry such as regulous functions. 
\end{abstract}
\maketitle

\section{Introduction}
\label{sec:intro}


This paper develops the geometry of locally bounded rational functions on real
algebraic varieties. If $R$ is a real closed field and $X \subseteq R^n$ is an
irreducible, non-singular algebraic variety, then a rational function $f$
defined on a Zariski dense subset of $X$ is locally bounded if its values are
bounded in some open neighbourhood of each point of $X$.  These functions have
already been studied in the literature in the guise of \emph{Real holomorphy
rings} (see \cite{BK82a, BK89, Kuc91, KR96, Sch82a, Sch03, Mon98}), albeit from
a completely algebraic point of view. Locally bounded rational functions
have also appeared in an analytic context
in the guise of arc-meromorphic functions in \cite{KP12}.

Locally bounded rational functions appear naturally in a geometric
context.  For example, the regular functions on the normalization of a
given singular real algebraic variety are locally bounded. In the
complex case, such functions on a normal variety are automatically
regular (by Hartog's Extension Theorem \cite[C 1.11]{Loj}),
however when working with real algebraic varieties one has many more of these
functions, a typical example being the function $(x, y) \mapsto
\frac{x^2}{x^2 + y^2}$
on $\R^2$. If the condition of local boundedness is replaced by continuity,
one obtains the class of continuous rational functions, which are
called \emph{regulous functions} if their domains are non-singular algebraic
varieties (cf. \cite{KolNow, KucKur}). 
The intent of this paper is to
study the ring of locally bounded rational functions on a non-singular real
algebraic variety while highlighting their similarities to,
and differences from regulous functions. It is important to note here that
the study of the behaviour of locally bounded rational functions on
\emph{singular} real algebraic varieties remains a topic for future work
and is excluded from this paper.

Locally bounded rational functions can be characterized in three equivalent ways
(Propositions \ref{prop:arcs}, \ref{prop:blow-regular} and \ref{prop:compact}):
(1) As those rational functions which map each semi-algebraic continuous arc in
$X$ to a bounded set in $R$, (2) which can be made regular with values in $R$
after the application of a sequence of blowups with smooth centres to $X$, and
(3) which map the intersection of each closed and bounded subset of $X$ and
their domain of definition to a bounded subset of $R$. In fact, the ring of
bounded rational functions is exactly the same as the ring of rational functions
which can be made regular with values in $R$ after a sequence of blowups with
smooth centres by utilizing the result of Hironaka \cite{Hir64}. Further, 
the ring of locally bounded rational functions on 
an irreducible and non-singular algebraic variety $X$ 
is non-Noetherian (Proposition \ref{prop:nnoetherian})
and has Krull dimension equal to the dimension of the underlying variety 
(Theorem \ref{thm:krull2}). 
This last result is an improvement over the previous estimate for the 
Krull dimension of this ring, given in 
\cite{BK82a}, which only showed that it was less than or equal to the 
dimension of the underlying variety. 

As a consequence of boundedness, the codimension of the locus of indeterminacy
of a locally bounded rational function on an irreducible smooth algebraic variety $X$ is
at least two (Theorem \ref{thm:codim}). 
This is similar to the regulous functions studied in, for example,  
\cite{Fic16}.  
Unlike for the case of these functions,
however, in order to define the zero set of a bounded rational function one must
resort to taking the limits of arcs or the image of its regularisation via a
sequence of blowups. This leads to a non-Noetherian (see Example \ref{ex:noetherian}) 
topology defined by these sets that is 
finer than that associated with rational continuous functions 
(see Examples \ref{ex:segment} and 
\ref{ex:semi-line}). The
differences do not end here however. In order to define the zero set of a
collection of locally bounded rational functions, it is necessary to consider
these functions as functions on arc-spaces of semi-algebraic 
continuous arcs. Another important property that these
functions have in common with regulous functions is 
the existence of a
\L{}ojasiewicz-type inequality (Theorems \ref{thm:lineqmain}, 
\ref{thm:lineqarc} and \ref{thm:lineqarc2}).  

In dimensions greater than or equal to 3, the set of locally bounded rational
functions that are zero on a given subset of $X$ may not be an ideal. In
dimension 2 however, as a direct consequence of the fact that the locus of
indeterminacy of a locally bounded rational function is of codimension 2 at least, and
hence consists only of isolated points, it is possible to construct the usual
algebro-geometric dictionary that one expects from other classes of functions 
(such as, for example, the regulous functions)
and recover results such as the Nullstellensatz (Theorem \ref{thm:nullsz1}). 

This paper is organized as follows: 
Section \ref{sec:background} will present some background and tools 
that will be used frequently throughout the paper. After that 
section \ref{sec:lbrf} concerns various algebraic properties of 
locally bounded rational functions. Section \ref{sec:geometry} will develop the 
geometry of locally bounded rational functions including their zero-sets. This 
will include some of the main results of the paper such as the 
various formulations of \L{}ojasiewicz-type inequalities.
Sections \ref{sec:arcspaces} and \ref{sec:dimension2} will 
be concerned with the reformulation of the notion of 
zero sets in terms of arc spaces of semi-algebraic arcs and 
the establishment of the usual algebro-geometric correspondence 
between these zero sets and ideals in the case of dimension 
2 respectively.

\subsection*{Acknowledgements}
\label{sec:ack}

The authors have received support from the Henri Lebesgue Center ANR-11-LABX-0020-01 and the project ANR New-Mirage ANR-23-CE40-0002-01.

\section{Background}
\label{sec:background}

\subsection{Notation and basics}
\label{sec:notation}

In what follows $R$ will denote a real closed field.  Let $X \subseteq R^n$ be a
non-singular, irreducible algebraic variety, in the sense of \cite{BCR13}. 
The ring of polynomial functions on $X$ will be denoted by $\Pol{X}$. 
For polynomials $p$ and $q$ in $\Pol{X}$, the quotient $f = p/q$ will 
be called a \emph{rational function} on $X$. These functions form a field 
which will be denoted by $R(X)$. If $f = p/q$ is a rational function and 
$q(x) \neq 0$ for all $x \in X$ then $f$ is called \emph{regular}. The 
set of regular functions on $X$ is a ring and will be denoted by $\Reg{X}$. 
The zero set of a function $f \in \Reg{X}$ or 
$\Pol{X}$ will be denoted by $\Z(f)$ and is 
called a \emph{Zariski closed set}. The complement of a 
Zariski closed set is called a \emph{Zariski open set}. 

In general an arbitrary $f \in R(X)$, where $f = p/q$ for relatively 
prime polynomials $p$ and $q$, 
is a function from a dense Zariski open subset of
$X$ to $\Proj{R}$, as there may be points $x \in X$ where 
$q(x) = 0$. However, there always exists a maximal dense Zariski open subset 
$U$ of $X$ such that $f\vert_U$ is regular. Such a $U$ is called the 
\emph{domain} of $f$ and will be denoted by $\domain(f)$. The set 
$X\setminus \domain(f)$ is called the \emph{locus of indeterminacy} of $f$ and 
will be denoted by $\indet(f)$. To emphasize this point a rational function 
from $X$ to $\Proj{R}$ will be denoted by $f: X \dashrightarrow \Proj{R}$. 

A \emph{semi-algebraic subset} of $R^n$ is a subset of the form\\ 
$\{x \in R^n | p_1(x) \geq 0, \dots, p_k(x) \geq 0\}$ where, 
$k \in \N$ and $p_i \in \Pol{R^n}$ for $0 \leq i \leq k$. 
An ideal $I$ of a ring $A$ is called \emph{real} if 
$f_1^2 + \dots + f_k^2 \in I$ implies $f_1 \in I$, 
where $f_1, \dots, f_k$ are elements of $A$. Further, 
if $f_1, \dots, f_k$ are elements of $A$ the notation, 
$\Ang{ f_1, \dots, f_k }$ will be used for the ideal 
generated by them. 

The graph of a map $f:X \rightarrow Y$, where $X \subseteq R^m$ and 
$Y \subseteq R^n$ for some integers $m, n$ will be denoted by 
$\graph{f}$ and is a subset of $X \times Y$. 

\subsection{Hironaka's resolution of singularities.} 
\label{sec:hironaka}

The following results will be used frequently throughout the paper. 
The first is a direct consequence of Hironaka's resolution of 
singularities \cite{Hir64}. 

\begin{theorem}[{cf. \cite{Hir64}}]
    \label{thm:hironaka}
    If $f: X \dashrightarrow \Proj{R}$ is a rational function on 
    an real, non-singular, irreducible algebraic variety $X$, then 
    there exists a composition of blowups with 
    smooth centres $\phi: \Tld{X} \rightarrow X$ such that 
    $\phi$ is an isomorphism between a dense open subset of 
    $\Tld{X}$ and a dense open subset of $X$ and 
    $\indet(f \circ \phi) = \varnothing$. 
\end{theorem}
In the remainder of this paper a composition of blowups with smooth 
centres will be called a \emph{resolution}. This next result 
follows immediately Theorem \ref{thm:hironaka} by taking the composition
of multiple resolutions. 
\begin{corollary}
    \label{cor:hironaka}
    If $X$ is a non-singular, irreducible algebraic variety, and 
    $f_1, \dots, f_k \in R(X)$, then there exists a resolution 
    $\phi: \Tld{X} \rightarrow X$ such that for all $i$, 
    $\indet(f_i \circ \phi) = \varnothing$. 
\end{corollary}

\subsection{The curve selection lemma}
\label{sec:curveselection}

\begin{theorem}[{The Curve Selection Lemma \cite[2.5.5]{BCR13}}]
    \label{thm:curveselection}
    Let $A \subseteq R^n$ be a semi-algebraic subset of $R^n$ and 
    let $x \in \overline{A}$. There exists a continuous semi-algebraic
    function $f : [0, 1] \rightarrow R^n$ such that 
    $f(0) = x$ and $f((0, 1]) \subseteq A$. 
\end{theorem}

\subsection{Puiseux series and arc spaces}
\label{sec:puiseux}
\emph{The field of Puiseux series} on a real closed field $R$ in 
an indeterminate $T$ is 
the set of formal series of the form, 
\begin{equation*}
    a = \sum_{i \geq m} a_i T^{i/n}, 
\end{equation*}
where $m \in \mathbb{Z}$, $n \in \N$ and $a_i \in R$ for all $i \geq m$. 
This field will be denoted by $\Puis{T}$. The set of elements of 
$\Puis{T}$ that are algebraic over the field of fractions of $R[X]$, 
will be denoted by $R\Ang{T}$. 

If $a \in \Puis{T}$, then the smallest exponent in the series corresponding 
to $a$ is called \emph{the order of $a$}. 
By convention the order of $0$ is taken to be $+\infty$. The set of elements of 
$R\Ang{T}$ with positive order are called \emph{bounded Puiseux series}, 
and is denoted by $R\Ang{T}_b$. 

\begin{proposition}
    \label{prop:puissemialg}
    The ring $R\Ang{T}$ is isomorphic to the set of germs at zero of 
    continuous semi-algebraic functions $[0, 1] \rightarrow R$, and 
    $R\Ang{T}_b$ is isomorphic to those germs that are 
    bounded. 
\end{proposition}
Let $P_1, \dots, P_k \in R[X_1, \dots, X_n]$, and, 
\begin{equation*}
    X = \{x \in R^n | P_i (x) = 0, 0 \leq i \leq k \}.
\end{equation*}
In this paper, the term \emph{semi-algebraic arc of $X$} 
will be used interchangeably for the following three objects: 
\begin{itemize}
    \item[1.] A continuous semi-algebraic 
        function $\gamma:[0, 1] \rightarrow X$. 
    \item[2.] The germ at 0 of a continuous semi-algebraic function 
        $R \rightarrow X$. 
    \item[3.] An $n$-tuple 
        $\gamma = (\gamma_1, \dots, \gamma_n) \in R\Ang{T}_b$ 
        such that $P_i(\gamma) = 0$ for $0\leq i \leq k$. 
\end{itemize}
Further the set, 
\begin{equation*}
    \Arc{X} = \{\gamma \in (R\Ang{T}_b)^n | P_i(\gamma) = 0, 0\leq i \leq k\} \setminus R^n 
\end{equation*}
will be called the \emph{arc space of $X$}. 
\begin{remark}
    \leavevmode\\
    \begin{itemize}
        \item [(i)] In the definition of $\Arc{X}$, $R$ is identified with the
            subset of
            $R\Ang{T}_b$ consisting of Puiseux series with only a term of order
            0, and removing $R^n$ ensures that we exclude all constant
            semi-algebraic arcs.
        \item [(ii)] With the above definition 
            $X = \{ x \in R^n | x = \lim_{t \to 0} \gamma (t), \exists \gamma \in \Arc{X}\}$. 
    \end{itemize}
\end{remark}

\section{Locally bounded rational functions}
\label{sec:lbrf}

\subsection{Locally bounded rational functions}
\label{sec:lbrf2}

Let $X \in R^n$ be an irreducible, non-singular algebraic variety. 
A rational function $f \in R(X)$ will be called a \emph{locally bounded 
rational function} on $X$ if for every $x \in X$ there
exists a euclidean neighbourhood $V_x$ of $x$ such that $f(V_x \cap \domain(f))$ 
is a bounded
subset of $R$. The set of all locally bounded rational functions on $X$ will
denoted by $\Rrat{X}$. 



By the following Lemma, the property of being locally rationally bounded
can be verified on \emph{any} dense Zariski open subset $U \subseteq X$, such
that $U \subseteq \domain(f)$. 

\begin{lemma}
    \label{lem:dense}
    Let $X \subseteq R^n$ be an irreducible, non-singular algebraic variety, 
    and $g \in R(X)$ be a rational function such that $g$ is regular 
    on a dense Zarsiki open subset $W \subseteq X$. If 
    there exists $f \in \Rrat{X}$ and a dense Zariski open subset $U \subseteq W$, 
    such that $f\vert_{U} = g\vert_{U}$, then $g \in \Rrat{X}$. 
\end{lemma}
\begin{proof}
    Let $x \in X$. By the hypothesis, there exists a neighbourhood $V_x
    \subseteq X$ of $x$ and $M \in R$ such that, $|f(y)| \leq M$ for all $y \in
    V_x \cap U$. Suppose now that $w \in W \cap V_x$. If $w \in U$ then
    $|g(w)| = |f(w)| \leq M$. If $w \notin U$, then as $g$ is continuous at $w$,
    and hence, there exists $\eta$ such that,
    \begin{equation}
        |w - z| \leq \eta \implies |g(w) - g(z)| \leq 1
    \end{equation}
    As $U$ is dense in $X$ (in the euclidean topology), there exists $z \in U\cap V_x$ such that $|w - z|
    \leq \eta$.  For such a $z$,
    \begin{equation}
        |g(w)| \leq |g(z)| + 1 \leq M + 1.
    \end{equation}
    Hence, $g(W \cap V_x)$ is bounded, and $g$ is a locally bounded rational 
    function. 
\end{proof}

\begin{example}
    \label{ex:protoyptical}
    Let $f \in \Rrat{R^2}$ be given by, 
    \begin{equation*}
        f (x, y)  = \frac{x^2}{x^2 + y^2}. 
    \end{equation*}
    Then $\indet(f) = \{ (0, 0) \}$ however it 
    is bounded by 1 in any neighbourhood of the origin as $|x^2 + y^2| > |x^2|$  
    for all $(x, y) \neq (0, 0)$. 
    Further if we let $y = ax$ 
    then we observe that, 
    \begin{equation*}
        \lim_{x \to 0} \frac{x^2}{(1 + a^2)x^2} = 1/(1 + a^2), 
    \end{equation*}
    which is bounded between $1$ (corresponding to $a = 0$) and 
    $0$ (corresponding to $a = \infty$). If we consider these 
    limits as $(x, y) \to (0, 0)$ along $y = ax$ as "values" of 
    $f$ then the above shows that $f$ takes on all values 
    between $1$ and $0$ at $(0, 0)$. A more formal definition 
    of the image of a locally bounded rational function will 
    be given later in Section \ref{sec:images}.  

    This function 
    will serve as a prototypical example of a locally bounded 
    rational function in this paper. 
\end{example}
The set of all locally bounded rational functions forms a ring. 
\begin{proposition}
    \label{prop:ringprop1}
    If $X \subseteq R^n$ be an irreducible, non-singular algebraic variety, then
    $\Rrat{X}$ is a subring of $R(X)$, the field of rational functions on $X$.
\end{proposition}
\begin{proof}
    If $f$ and $g$ are bounded by $M$ and $N$ on the neighbourhoods $V_x$ and
    $W_x$ then $f + g$ and $fg$ are bounded by $M + N$ and $MN$ respectively on
    $V_x \cap W_x$.  Similarly, $-f$ is bounded by $M$ on $V_x$.
\end{proof}
The following is obvious. 
\begin{corollary}
    \label{cor:ringprop1}
    If $X \subseteq R^n$ is an irreducible, non-singular algebraic variety then
    $\Rrat{X}$ is an integral domain.
\end{corollary}
Locally bounded rational functions are characterized by the fact that they are
exactly those rational functions that map semi-algebraic arcs in their domain to
bounded subsets of $R$. This characterization will be of immense utility in what
follows:
\begin{proposition}
    \label{prop:arcs}
    Let $X \subseteq R^n$ be an irreducible, non-singular algebraic variety, and 
    $f \in R(X)$. Then, $f \in \Rrat{X}$ if and only if for every 
    semi-algebraic and continuous arc $\gamma : [0, 1] \rightarrow X$, 
    such that $\gamma ((0, 1]) \subseteq \domain(f)$, the 
    function $f \circ \gamma : (0, 1] \rightarrow R$ is bounded. 
\end{proposition}
\begin{proof}
    For the "if" direction, 
    let $U = \domain(f)$. 
    Suppose there exists $x \in X$ such that for every neighbourhood $V_x$ of
    $x$, $f(V_x\cap U)$ is not bounded. Fix $\eta \in R_{>0}$, and let $B(x,
    \eta)$ be the open ball centred at $x$ with radius $\eta$. There exists
    $y_{\eta} \in B(x, \eta)$ such that $f(y_{\eta}) \geq 1/\eta$. Let $A
    \coloneqq \{ (x, y) \in U \times R^* | |f(x)| \geq 1/y\} \subseteq X \times
    R$.  This set is semi-algebraic. Now, if $\epsilon > 0$, then for all $ \eta
    < \epsilon /\sqrt{2}$,
    \begin{equation*}
        \sqrt{ \|y_{\eta} - x\| + \eta^2 } \leq \sqrt{ \eta^2 + \eta^2 } < \epsilon. 
    \end{equation*}
    Taking $\epsilon \rightarrow 0$, this implies that 
    $(x, 0) \in \overline{A}$. Now, by the curve selection
    lemma (Theorem \ref{thm:curveselection}), there exists a 
    semi-algebraic, continuous arc
    inside $A$ which approaches $(x, 0)$ in the limit. The first $n$ coordinate
    functions of this arc define an arc $\gamma:(0, 1] \rightarrow X$ whose
    image lies within $U$. This arc is semi-algebraic as it is the projection of
    a semi-algebraic arc and by construction $(f \circ \gamma) ((0, 1])$ is not
    bounded.

    Now, for the "only if" direction of the argument, 
    suppose that $\gamma : [0, 1] \rightarrow X$ is a semi-algebraic,
    and continuous arc such that $\gamma((0, 1]) \subseteq U$. Let $x =
    \gamma(0)$. By the hypothesis, there exists a neighbourhood $V_x$ of $x$
    such that $f(V_x \cap U) \subseteq [-M, M]$ for some $M \in R$. Without loss
    of generality one may assume that $V_x$ is open.  Now, as $\lim_{t \to 0}
    \gamma (t) = x$, there exists $\epsilon > 0$ such that $\gamma([0,
    \epsilon)) \subset V_x$. This implies that $f(\gamma( ( 0, \epsilon)))$ is
    bounded by $M$. Now, since $\gamma([\epsilon, 1]) \subseteq U$,
    $f(\gamma([\epsilon, 1])$ is bounded as it is the image of a closed and
    bounded set by a continuous map. Therefore, $(f\circ \gamma) ((0, 1])$ is
    bounded.
\end{proof}
The following is an easy corollary of Proposition \ref{prop:arcs}, which 
shows that a locally bounded rational function 
$f$ can be given "values" on points lying inside $\indet(f)$, 
by taking the limits of its values along continuous semi-algebraic 
arcs terminating at these points. 
\begin{corollary}
    \label{cor:indentlocus}
    Let $X \subseteq R^n$ be an irreducible, non-singular, algebraic variety, 
    and $f \in \Rrat{X}$. For any $x \in \indet(f)$, there exists a 
    continuous, semi-algebraic arc $\gamma:[0, 1] \rightarrow X$, with 
    $\gamma((0, 1]) \subseteq \domain(f)$, such that, 
    \begin{equation*}
        \lim_{t \to 0} \gamma(t) = x,
    \end{equation*}
    and $\lim_{t\to 0} (f \circ \gamma) (t) < \infty$. 
\end{corollary}
\begin{proof}
    The existence of $\gamma$ such that 
    $\lim_{t \to 0} \gamma(t) = x$ is a consequence of the fact that 
    $\domain(f)$ is a dense Zariski open subset of $X$ and hence, 
    $x \in \indet(f)$ implies that $x \in \overline{\domain(f)}$, 
    and the curve selection lemma (cf. Theorem \ref{thm:curveselection}). 
    Now, by Proposition \ref{prop:arcs} $f \circ \gamma$ is bounded, 
    and by \cite[Proposition 2.5.3]{BCR13}, can be extended continuously 
    to zero, implying
    that $\lim_{t\to 0} (f \circ \gamma)(t) < \infty$. 
\end{proof}

\subsection{Locally bounded rational maps}
\label{sec:ratmapsinit}

A rational map $f : R^m \dashrightarrow \Proj{R^n}$ 
is called a \emph{locally bounded rational
map} if all its coordinate functions are locally bounded rational functions.
This definition is similar for locally bounded rational maps between two
irreducible, non-singular real algebraic varieties $X$ and $Y$. The set of all
locally bounded rational maps from $X$ to $Y$ is denoted by $\Rrat{X, Y}$.

\subsection{Locally bounded functions are blow-regular}
\label{sec:blowup}

The objective of this section is to show that for an irreducible, non-singular 
algebraic variety $X$, the ring of locally bounded functions coincides with 
the set of rational functions which can be made regular with 
values in $R$ after an application of a composition of blowings up with 
smooth centres to $X$. This second characterisation can be used to prove 
another characterisation of locally bounded rational functions in terms of their
action on closed and bounded subsets of $X$.  
\begin{proposition}
    \label{prop:blow-regular}
    Let $X \subseteq R^n$ be an irreducible, non-singular algebraic variety. If
    $f \in \Rrat{X}$ then there exists a composition of blowups with smooth
    centres $\phi:\Tld{X} \rightarrow X$ such that $f \circ \phi : \Tld{X}
    \rightarrow \Proj{R}$ is regular and such that $(f \circ \phi)
    (\Tld{X}) \subseteq R$.
\end{proposition}
\begin{proof}
    By Theorem \ref{thm:hironaka} there exists a
    composition of blowups $\phi:\Tld{X} \rightarrow X$ with smooth centres such
    that $f \circ \phi : \Tld{X} \rightarrow \Proj{R}$ is regular.

    Suppose that there exists $\Tld{x} \in \Tld{X}$ such that $\Tld{f}(\Tld{x})
    = \infty$, where $\Tld{f} = f \circ \phi$. Also, observe that $\Tld{U}
    \coloneqq \phi^{-1} (\domain f)$ is dense (in the euclidean topology) in
    $\Tld{X}$. Therefore, by the curve selection lemma 
    (Theorem \ref{thm:curveselection}) there
    exists a semi-algebraic arc $\Tld{\gamma}:[0,1] \rightarrow \Tld{X}$ such
    that $\Tld{\gamma}(0) = \Tld{x}$ and $\Tld{\gamma}((0,1]) \subseteq
    \Tld{U}$. Let now, $\gamma = \phi \circ \Tld{\gamma}$.  This is a
    semi-algebraic arc that satisfies $\lim_{t\to 0} \gamma(t) = \phi(\Tld{x})$,
    and
    \begin{equation*}
        \lim_{t\to 0}(f\circ \gamma) (t) = \infty.
    \end{equation*}
    This implies, by Proposition \ref{prop:arcs} that $f$ is not a locally   
    bounded function.
\end{proof}
With the above theorem it is now possible to give an alternative 
characterization of locally bounded rational functions in terms of 
their action on closed and bounded sets (compact sets in the case 
when $R = \R$).  
\begin{proposition}
    \label{prop:compact}
    Let $X \subseteq R^n $ be an irreducible, non-singular algebraic variety, $U
    \subseteq X$ be a Zariski open subset of $X$ and $f: U \rightarrow R$ be a
    rational function on $X$. The function $f$ is locally bounded if and  only
    if for every closed and bounded subset $K$ of $X$, the set $f(K \cap U)$ is 
    a bounded subset of $R$.
\end{proposition}
\begin{proof}
    The "only if" part of the proposition follows directly from the 
    definition of a locally bounded rational function. 
    For the reverse implication, let $K$ be a closed and bounded set, 
    $f \in R_{b}(X)$ 
    and $\Tld{K}$ be the inverse image of $K$ in 
    $\phi: \Tld{X} \rightarrow X$ that makes $f$ regular by 
    Proposition \ref{prop:blow-regular}. Then, 
    \begin{equation*}
        f(K \cap U) = \Tld{f}(\Tld{U} \cap \Tld{K}) \subseteq \Tld{f}(\Tld{K}). 
    \end{equation*}
    Here $\Tld{K}$ is a closed and bounded set as it is the inverse 
    image of a bounded set in the proper map $\phi$.  
    The result then follows from the fact that $\Tld{f}(\Tld{K})$ is bounded as 
    it is the continuous image of a closed and bounded set. 
\end{proof}

\begin{theorem}
    \label{thm:blowupsmain}
    If $f \in R(X)$ where $X \subseteq R^n$ is an irreducible, non-singular
    algebraic variety and $f$ becomes a regular function with values in $R$
    after a sequence of blow-ups then $f$ is locally bounded.
\end{theorem}
\begin{proof}
    Let $f \in R(X)$ be a rational function that becomes regular with
    values in $R$ after 
    a sequence of blow-ups $\phi: \Tld{X} \rightarrow X$ and let 
    $\Tld{U} = \phi^{-1}(U)$, where $U = \domain(f)\subseteq X$. By \cite{Hir64}, 
    the map $\phi$ is an isomorphism between $\Tld{U}$ and $U$. Further, let
    $\Tld{f} = f \circ \phi$. 

    Now, if $f$ is not a locally bounded rational function, then, by Proposition
    \ref{prop:arcs}, there exists a semi-algebraic arc $\gamma:[0, 1]
    \rightarrow X$ such that $\gamma((0,1]) \subseteq U$ and $\lim_{t
    \to 0}\gamma(t) = x \in X$ such that $f \circ \gamma$ is not bounded. Let
    $\Tld{\gamma} = \phi^{-1} \circ \gamma : (0, 1] \rightarrow \Tld{U}$.
    $\Tld{\gamma} ((0, 1])$ is bounded because $\phi$ is a proper map and may be
    extended by continuity to $0$ (cf. \cite[Proposition 2.5.3]{BCR13}). 
    If $\Tld{x} = \lim_{t \to 0} \Tld{\gamma}(t)$ then,
    \begin{equation*}
        \lim_{t \to 0} ( \Tld{f} \circ \Tld{\gamma} )(t) = \lim_{t \to 0} (f \circ \gamma) (t) = \infty.
    \end{equation*}
    Therefore, $\Tld{f}(\Tld{x}) = \infty$ and all the values of
    $\Tld{f}$ do not lie in R. 
\end{proof}
\begin{remark}
    \label{rem:blowregular}
    Theorem \ref{thm:blowupsmain} was proved in \cite{KR96} for the 
    case when $R = \R$, however the proof for a general real closed 
    field presented above is almost identical. 
\end{remark}
\begin{theorem}
    \label{thm:biratiso}
    Every birational proper morphism $\phi:\Tld{X} \rightarrow X$ 
    between two affine, non-singular and irreducible 
    algebraic varieties, $\Tld{X}$ and $X$ over $R$ induces an isomorphism
    between $\Rrat{X}$ and $\Rrat{\Tld{X}}$ given by $f \mapsto f \circ \phi$.  
\end{theorem}
\begin{proof}
    Let $f \in R_{b}(X)$ with $\domain(f) = U$, $\Tld{U} = \phi^{-1} (U)$, and
    $\Tld{\gamma}$ be a semi-algebraic arc in $\Tld{X}$ such that
    $\Tld{\gamma}((0, 1]) \subseteq \Tld{U}$.  If $\gamma = \phi \circ
    \Tld{\gamma}$, then $\gamma$ is a semi-algebraic arc with $\gamma((0, 1])
    \subseteq U$. As $f$ is a locally bounded rational function $(f\circ
    \gamma) ([0,1])$ is bounded. Since $\Tld{f} \circ \Tld{\gamma} = f \circ
    \gamma$, $\Tld{f}$ is a locally bounded rational function 
    by Proposition \ref{prop:arcs}.

    Now, let $\Tld{f} \in R_{b}(\Tld{X})$, and $K$ be a closed and bounded
    subset of $X$.  As $\phi$ is a proper map $\Tld{K} \coloneqq \phi^{-1}(K)$
    is closed and bounded set of $\Tld{X}$. Observe that $f(K) \subseteq
    \Tld{f}(\Tld{K})$, and $\Tld{f}(\Tld{K})$ is bounded by Proposition
    \ref{prop:compact}, which implies that $f(K)$ is bounded, which in turn
    implies that $f \in R_{b}(X)$.
\end{proof}

\subsection{Properties of locally bounded rational functions}
\label{sec:propfuncs}

This section develops certain properties of locally bounded rational 
functions and maps. These include their relationship to regulous 
functions, the codimension of their loci of indeterminacy and 
the integral closedness of $\Rrat{X}$ in the field of rational 
functions $R(X)$ where $X$ is an irreducible, non-singular algebraic
variety. It will also be shown that $\Rrat{X}$ is a non-Noetherian 
ring and that $\dim \Rrat{X} = \dim X$.  

\begin{proposition}
    \label{prop:contrat}
    If $f = p/q \in \Rrat{X}$ where  
    $X \subseteq R^n$, is an irreducible non-singular 
    algebraic variety then $g = p^2/q$ is a regulous function. 
\end{proposition}
\begin{proof}
    Let $x \in \Z(q)$. Since $f$ is locally bounded at $x$, $p(x) = 0$.
    Further, let $V_x$ be an open neighbourhood of $x$ in $X$. By definition
    there exists $M \in R$ such that $|f(x)| \leq M$ for all $x \in \domain(f)
    \cap V_x$. As $p$ is continuous and $p(x) =0$, for each $\epsilon > 0$ there
    exists a neighbourhood $W_x \subseteq V_x$ of $x$ such that $|p(x)| \leq
    \epsilon /M$ for all $x \in W_x$. Therefore, $|g(x)| = |p(x) \cdot f(x)|
    \leq M \cdot (\epsilon/M) = \epsilon$ for all $x \in W_x \cap
    \domain(f)$. Therefore, $g$ is continuous at $x$.
\end{proof}

\begin{theorem}
    \label{thm:codim}
    If $f \in \Rrat{X}$ for an irreducible, non-singular algebraic variety $X
    \subseteq R^n$, with ideal of definition $I$ and $f = p/q$ where $p, q \in
    R[x_1, \dots, x_n]/I$ are two relatively prime polynomials over $R$, then
    $\Z(q) \subseteq \Z(p)$ and $\codim_X \Z(q) \geq 2$.
\end{theorem}
\begin{proof}
    If $\Z(q) \not \subseteq \Z(p)$ then $f = p/q$ would not be bounded,
    therefore $\Z(q)$ must be a subset of $\Z(p)$.  Suppose that $\codim_X \Z(q)
    \leq 1$. As $q$ is not identically zero, this implies that $\codim_X \Z(q) =
    1$. Therefore there is a divisor $q'$ of $q$ such that 
    $q'$ is irreducible and $\codim_X \Z(q') = 1$.
    Now by \cite[Theorem 4.5.1]{BCR13} the ideal generated by $q'$ in 
    $R[x_1, \dots, x_n]/I$ is a real ideal. As $q'$ is irreducible this ideal is
    also radical. The inclusions $\Z(q') \subseteq \Z(q) \subseteq \Z(p)$ 
    imply that $p \in \I ( \Z(q') )$ which, in turn, implies that 
    $q'$ divides $p$ contradicting the hypothesis that $p$ and $q$ are 
    relatively prime. 
\end{proof}
The following are immediate consequences of the above result. 
\begin{corollary}
    \label{cor:dim1}
    A locally bounded rational function on a non-singular, irreducible 
    real algebraic variety $X$ with $\dim (X) = 1$ is regular. 
\end{corollary}
\begin{corollary}
    \label{cor:dim21}
    A locally bounded rational function on a non-singular, irreducible 
    real algebraic variety $X$ with $\dim (X) = 2$ is regular everywhere 
    except at a finite number of points.
\end{corollary}

\begin{corollary}
    \label{cor:mapscomp}
    Let $f \in \Rrat{X, Y}$ and $g \in \Rrat{Y, Z}$, where $X, Y, Z$ are,
    irreducible, non-singular algebraic varieties over $R$. If $\codim
    (\Image(f)) \leq 1$ then $f\circ g \in \Rrat{X, Z}$.
\end{corollary}
The following lemma is a direct consequence of the characterisation of a 
locally bounded rational functions by arcs (Proposition \ref{prop:arcs}). 
\begin{lemma}
    \label{lem:comp}
    Let $f \in \Rrat{X, Y}$ and $g \in \Rrat{Y, Z}$, where $X, Y, Z$ are
    irreducible, non-singular algebraic varieties. If $f(\domain(f)) \not
    \subseteq \indet(g)$ then $f\circ g \in \Rrat{X, Z}$.
\end{lemma}

\begin{proposition}
    \label{prop:radreal}
    If $I \subseteq \Rrat{X}$ is a radical ideal then it is real. 
\end{proposition}
\begin{proof}
    Suppose that  $f^2_1 + \dots + f^2_k \in I$. 
    For each $ 1 \leq i \leq k$, $f^2_i/(\sum_{j = 1}^k f^2_j) \in I$ 
    by Lemma \ref{lem:comp} (Consider the composition of $F = (f_1, \dots, f_k)$ and 
    $G = (g_1, \dots, g_k)$, where $g_i = x^2_i/(\sum_{j = 1}^k x^2_j)$). Hence 
    $f^2_i \in I$ because, 
    \begin{equation*}
        f^2_i = (f^2_1 + \dots + f^2_k)\frac{f^2_i}{\sum_{j = 1}^{k} f^2_j} \in I. 
    \end{equation*}
    Now as $I$ is radical $f_i \in I$ (cf. \cite[Lemma 4.1.5]{BCR13}). 
\end{proof}
The following proposition is a consequence of the fact that a composition of
blowups with smooth centres has the arc-lifting property for analytic arcs
(\cite{FP08}).
\begin{proposition}
    \label{prop:analyticarc}
    If $f \in \Rrat{X}$ where $R = \R$, and $\gamma : (-\epsilon, \epsilon)
    \rightarrow X$ is an analytic arc such that $\gamma((-\epsilon, 0) \cup (0,
    \epsilon)) \subseteq \domain (f)$, then $f\circ \gamma$, extended by continuity
    at $0$ is also an analytic arc.
\end{proposition}
The following proposition uses an adaptation of 
a counter example due to Kurdyka 
from \cite{Karc} to show 
that the ring of locally bounded rational functions is not Noetherian. 
\begin{proposition}
    \label{prop:nnoetherian}
    The ring $\Rrat{R^n}$ is non-Noetherian for $n\geq 2$. 
\end{proposition}
\begin{proof}
    For $k \in \N$, let $f_k = x_1^2 / (x_1^2 + (x_2-k)^2)$ and $I_k \subseteq
    \Rrat{R^n}$ be the ideal generated by $f_1, \dots, f_k$. If for some $k$,
    $f_{k+1} \in I_k$, this implies that there exist $g_j \in I_k$ for $1\leq j
    \leq k$ such that $f_{k+1} = \sum_{j = 1}^k g_j f_j$. Let 
    $W = (\cap_{i = 1}^k \domain(g_i)) \cap (\cap_{i = 1}^k \domain(f_i)) \cap
    \domain(f_{k+1})$. Note that 
    $y = (0, k+1, 0, \dots, 0) \in \indet(f_{k+1}) \subseteq \overline{W}$ as 
    each of the sets in the definition of $W$ is a dense Zariski open set. 
    By the curve selection lemma (Theorem \ref{thm:curveselection}), there 
    exists a continuous semi-algebraic arc $\gamma: [0, 1] \rightarrow R^n$, 
    such that $\gamma((0, 1]) \subseteq W$ and $\lim_{t \to 0} \gamma(t) = y$. 
    Now by Corollary \ref{cor:indentlocus}, 
    the limits $\lim_{t\to 0} (f_i \circ \gamma)(t)$ 
    and $\lim_{t \to 0} (g_i \circ \gamma)(t)$, exist for all $i$ and therefore, 
    \begin{equation*}
        \lim_{t\to 0} (g_i \circ \gamma)(t) \cdot (f_i \circ \gamma)(t) = 
            \lim_{t \to 0} (g_i \circ \gamma)(t) \cdot \lim_{t \to 0} (f_i \circ \gamma)(t)  = 0\;\;\; \text{ for all } 1 \leq i \leq k,  
    \end{equation*} 
    by the definition of $f_i$ for $i \leq 0$, 
    while $\lim_{t\to 0} ( f_{k + 1}\circ \gamma )(t) = 1$ 
    Therefore composing with $\gamma$ and taking limits as $t \to 0$ on 
    both sides of the equation, 
    \begin{equation*}
        f_{k + 1} = \sum_{i = 1}^{k} f_i g_i 
    \end{equation*}
    one obtains $1 = 0$ which implies, by contradiction, that 
    $f_{k + 1} \not \in (f_1, \dots, f_k)$ for each $k \in \N$, and  
    therefore this sequence of ideals forms an infinitely long 
    ascending chain, and $\Rrat{R^n}$ cannot be Noetherian. 
\end{proof}

\begin{proposition}
    \label{prop:integralclosure}
    If $X$ is an irreducible, non-singular algebraic variety then
    $\Rrat{X}$ is integrally closed in $R(X)$. 
\end{proposition}
\begin{proof}
    Let $f$ be a rational function on $X$ such that, 
    \begin{equation}
        \label{eq:integralclosure1}
        f^n + g_{n-1} f^{n-1} + \dots + g_0 = 0
    \end{equation}
    for some $g_0, \dots, g_{n-1} \in \Rrat{X}$. 
    Let $x \in X$, then there exists a dense Zariski open set 
    $U$ and a neighbourhood $V_x$ of $x$, such that
    $g_0, \dots, g_{n-1}$ are bounded by some $M \in R$ on $V_x \cap U$.  
    Then for each $y \in V_x \cap U$, 
    equation \ref{eq:integralclosure1}, along with the triangle inequality 
    then yields:
    \begin{equation*}
        |f(y)|^n \leq M(|f(y)|^{n-1} + \dots + |f(y)| + 1).
    \end{equation*}
    If there exists $y_0 \in V_x \cap U$ such that $f(y_0) \geq 1$, then the 
    above implies that, 
    \begin{equation*}
        |f(y_0)|^n \leq M \cdot n (|f(y_0)|^{n-1})
    \end{equation*}
    which implies that $|f(y_0)| \leq M \cdot n$. Therefore, 
    $|f(y)| \leq \max \{ 1, M \cdot n \}$ for all $y \in V_x \cap U$, 
    which implies that $f \in \Rrat{X}$. 
\end{proof}
The next result that will be established is the Krull dimension of 
the rings $\Rrat{X}$. The following result from \cite[IV, \S 10]{ZS} will 
be used to bound the Krull dimension from above. 
\begin{proposition}
    \label{prop:krull1}
    Let $A, B$ be two commutative rings, and $A \subseteq B$ with rings of
    fractions $F$ and $K$ respectively. Then $\dim B \leq \dim A +
    \mathrm{tr}(K/F)$, where $\dim$ denotes the Krull dimension and
    $\mathrm{tr}$ denotes the transcendence degree of $K$ over $F$.
\end{proposition}
Proposition \ref{prop:krull1}, and
the fact that the rings of polynomials and locally bounded rational functions
on a non-singular, irreducible, algebraic variety have the same field of
fractions (i.e. the field of rational functions), immediately imply 
the following: 
\begin{corollary}
    \label{cor:krullineq}
    If $X$ is a non-singular, irreducible, algebraic variety of dimension $n$, 
    then $\dim \Rrat{X} \leq n$. 
\end{corollary}
\begin{theorem}
    \label{thm:krull2}
    The Krull dimension of $\Rrat{R^n}$ is $n$. 
\end{theorem}
\begin{proof}
    If $n = 0$ then $\Rrat{R^0} = R$, which has one prime ideal (the zero
    ideal).  If $n\geq 1$, let $\phi : \Rrat{R^n} \rightarrow \Rrat{R^{n-1}}$ be
    the map that sends $f \in \Rrat{R^n}$ to $(x_1, \dots, x_{n-1}) \mapsto
    f(x_1, \dots, x_{n-1}, 0)$.  This is the pullback of the canonical injection
    $R^{n-1} \xhookrightarrow{} R^n$. Since the image of this map is of
    codimension less than 1, by Corollary \ref{cor:mapscomp} $\phi$ is well
    defined. If $\dim \Rrat{R^{n-1}} \geq n-1$ then there exists a chain of prime
    ideals, $P_0 \subsetneq P_1 \subsetneq \dots \subsetneq P_{n-1}$ in
    $\Rrat{R^{n-1}}$.  As $\Rrat{R^{n-1}}$ and $\Rrat{R^n}$ are integral
    domains, the inverse images of these in $\phi$ form a chain of prime ideals
    of length $n-1$ in $\Rrat{R^n}$, which is strictly increasing because $\phi$
    is surjective. Since the kernel of $\phi$ is non-zero, $\phi^{-1}(P_0) \neq
    \Ang{ 0 }$, and hence adding the zero ideal to this chain of inverse images
    produces a chain of length $n$, which implies $\dim \Rrat{R^{n}} \geq n$.
    The result follows by induction.
\end{proof}
\begin{remark}
    Theorem \ref{thm:krull2} is a new result. The fact that 
    the dimension of $\Rrat{R^n}$ is bounded form above by $n$, is 
    established in \cite[Theorem 1.22]{BK82a}. 
    On the other hand a consequence of \cite[Theorem 1.21]{BK82a} is that 
    the ring of 
    rational functions that are \emph{uniformly} bounded, that is, 
    those which can be bounded on the whole of the domain 
    by a single element of $R$, has dimension equal 
    to the domain, however this is a strict subring of $\Rrat{R^n}$.  
\end{remark}

\section{The geometry of locally bounded rational functions}
\label{sec:geometry}

\subsection{Zero-set of a locally bounded rational function}
\label{sec:zero-set}

This section is primarily concerned with the geometry of zero sets of locally bounded 
rational functions and the topology associated with them. By Example 
\ref{ex:protoyptical}, locally bounded rational functions can be considered 
multi-valued at points that belonging to their loci of indeterminacy. Therefore, 
it is necessary to formulate a definition of the zero set of a locally 
bounded rational function without resorting to evaluation. One can 
conceive of many ways to define such an object, and this section 
presents three such definitions which are shown to be equivalent. 

If $f$ is a rational function on  
an irreducible, non-singular algebraic variety $X \subseteq R^n$, the set  
$\graph{f} \subseteq X \times R$ is the graph of $f$ and $U = \domain(f)$, then 
the following are three possible ways to define the zero set of $f$: 
\begin{itemize}
    \item[(1)] $\Z_{arc}(f) \coloneqq \{ x \in X | \exists \gamma:[0,1] \rightarrow X  \text{ semi-algebraic and continuous with } \\ \gamma(0) = x \text{ and } \gamma((0,1)) \subseteq U \text{ such that } \lim_{t\to 0} f(\gamma(t)) = 0  \}$.
    \item[(2)] $\Z_{res}(f) \coloneqq \{ x \in X| 
        \exists \text{ a resolution } \phi: \Tld{X} \rightarrow X, 
        \Tld{x} \in \Tld{X} \text{ s.t. } 
        \phi(\Tld{x}) = x, f\circ \phi (\Tld{x}) = 0\}$
    \item[(3)] $\Z_{graph}(f) \coloneqq \{x \in X | (x, 0) \in \overline{\graph{f}}\}$ 
\end{itemize}
As the following theorem shows (1), (2), and (3) above are equivalent. 
\begin{theorem}
    \label{thm:zero-set}
    If $f \in R(X)$ where $X \subseteq R^n$ is an irreducible, non-singular,
    algebraic variety, then,
    \begin{equation}
        \Z_{arc}(f) = \Z_{res}(f) = \Z_{graph}(f). 
    \end{equation}
\end{theorem}
\begin{proof}
    $\Z_{arc}(f) \subseteq \Z_{res}(f)$:\\
    Let $x \in \Z_{arc}(f)$, and let $U = \domain(f)$. 
    Further, let $\gamma:[0,1]\rightarrow X$ be a
    semi-algebraic, continuous arc such that $\lim_{t \to 0} \gamma(t) = x$,
    $\gamma((0, 1]) \subseteq U$ and $\lim_{t \to 0}(f\circ \gamma)(t) = 0$.
    If $\phi:\Tld{X} \rightarrow X$ is a resolution that makes $f$ regular 
    and $\Tld{f} = f \circ \phi$, then for $t \in (0, 1]$, the arc $\phi^{-1} \circ
    \gamma$ is well defined and continuous as the exceptional locus of $\phi$
    does not intersect $U$ and $\phi$ is an isomorphism outside its exceptional
    locus. If $K = \gamma([0,1])$, then $K$ is closed and bounded and hence
    $\Tld{K} = \phi^{-1}(K)$ is closed and bounded because $\phi$ is a 
    proper map. Since $\Tld{\gamma}((0,1])
    \subseteq K$, by \cite[Proposition 2.5.3]{BCR13}, 
    $\Tld{\gamma}$ may be extended by
    continuity to $0$. If $\Tld{x} = \lim_{t \to 0}(\Tld{\gamma}(t))$, then,
    \begin{equation*}
        \Tld{f}(\Tld{x}) = \lim_{t \to 0} \Tld{f} (\Tld{\gamma}(t)) = \lim_{t \to 0}f (\gamma(t)) = 0.
    \end{equation*}
    Therefore $x \in \Z_{res}(f)$. 

    $\Z_{res}(f) \subseteq \Z_{graph}(f)$:\\
    Let $\phi: \Tld{X} \rightarrow X$ be a resolution that makes $\Tld{f} = f
    \circ \phi: \Tld{X} \rightarrow R$ regular and let $U$ be $\domain (f)$.
    Let $\Phi = \phi \times \mathrm{Id} : \Tld{X} \times R
    \rightarrow X \times R$.  
    Then, $\Phi (\graph{\Tld{f}}) = \overline{\graph{f}}$,
    because the left hand side is closed, as $\Phi$ is proper and $X
    \times R$ is locally compact, and coincides with $\graph{f}$ on a dense
    subset.

    $\Z_{graph}(f) \subseteq \Z_{arc}(f)$:\\
    Suppose $(x, 0) \in \overline{\graph{f}}$. By the curve selection lemma
    (Theorem \ref{thm:curveselection}), there exists a semi-algebraic arc $\hat{\gamma}:[0, 1]
    \rightarrow \overline{\graph{f}}$ such that $\hat{\gamma}((0, 1]) \subseteq
    \graph{f}$ and $\hat{\gamma}(0) = (x, 0)$.  If $\gamma:[0, 1] \rightarrow X$
    is the curve defined by the first $n$ coordinates of $\hat{\gamma}$, then
    $\gamma((0, 1]) \subseteq U$ and $\lim_{t \to 0} f(\gamma(t)) = 0$, and
    hence $x \in \Z_{arc}(f)$.
\end{proof}
In light of the above theorem, $\Z(f)$ will be used to denote the sets (1), (2)
and (3), and will be called simply \emph{the zero set of $f$}. 
The resolution that makes a locally bounded rational function regular is 
not uniquely determined, which makes it necessary to show that the 
definition (2) does not change depending on the resolution chosen.
Note here that a resolution for $f \in \Rrat{X}$ has been defined to be a sequence of
blowings-up $\phi: \Tld{X} \rightarrow X$ that renders $f \circ \phi$ regular with
values in $R$. 
\begin{lemma}
    \label{lem:res-indep}
    Let $f \in \Rrat{X}$, then for every resolution $\phi$ such that $f \circ \phi$ is 
    regular, one has $\Z(f) = \phi (\Z(f\circ \phi))$. 
\end{lemma}
\begin{proof}
  Suppose that $\phi: \Tld{X} \rightarrow X$ is a resolution that renders
  $f \circ \phi$ regular and that $x \in \phi(\Z(f \circ \phi))$.
  Then there exists $\Tld{x}$ such that $\phi(\Tld{x}) = x$ and,
  $f \circ \phi (\Tld{x}) = 0$. Therefore $\phi (\Z(f \circ \phi) \subseteq \Z(f)$. 

  Now, for the other inclusion, note that in the definition of $\Z_{res}(f)$,
  a priori, the resolution $\phi$ depends on each point $x$, and the result is established
  by showing that for any two resolutions $\phi: \Tld{X} \rightarrow X$ and,
  $\theta: \hat{X} \rightarrow X$, a point $\hat{x} \in \hat{X}$ such that,
  $(f \circ \theta) (\hat{x}) = 0$, and $x = \theta(\hat{x})$ implies the
  existence of $\Tld{x} \in \Tld{X}$ such that $(f \circ \phi) (\Tld{x}) = 0$,
  and $\phi(\Tld{x}) = x$. This argument is presented below:
  
  Suppose $\phi:\Tld{X} \rightarrow X$ and $\theta : \hat{X} \rightarrow X$
  are two resolutions such that $\Tld{f} = f \circ \phi : \Tld{X} \rightarrow
  R$, and $\hat{f} = f \circ \theta : \hat{X} \rightarrow R$ are regular with
  values in $R$. Suppose also that $\hat{\gamma}:[0, 1] \rightarrow \hat{X}$
  is a semi-algebraic arc with $\hat{\gamma}((0, 1]) \subseteq \hat{U} =
  \theta^{-1} (\domain(f))$ and $\hat{\gamma}(0) = \hat{x}$, such that $\hat{f}
  (\hat{x}) = 0$. Then, $\Tld{\gamma} = \phi^{-1} \circ \theta \circ
  \hat{\gamma} : (0, 1] \rightarrow \Tld{X}$ is another semi-algebraic
  continuous arc, which can be extended to $0$ by continuity (using
  \cite[Proposition 2.5.3]{BCR13}) to obtain $\Tld{x} = \lim_{t \to 0} \Tld{\gamma}(t)$,
  with $\Tld{f}(\Tld{x}) = 0$.
\end{proof}

\subsection{Characterization by blowups}
\label{sec:blowupschar}

The following is an immediate consequence of Lemma \ref{lem:res-indep} and 
Theorem \ref{thm:zero-set}.
\begin{proposition}
    \label{prop:zerosetsblowups}
    If $F = \Z(f)$ for $f \in \Rrat{X}$, where 
    $X$ is an irreducible, non-singular algebraic variety, then 
    there exists a resolution
    $\phi: \Tld{X} \rightarrow X$ and $Z \subseteq \Tld{X}$, a closed Zariski 
    subset such that $\phi(Z) = F$. 
\end{proposition}

\begin{proposition}
    \label{prop:zerosetsblowups2}
    If $F \subseteq X$ is the image of a Zariski closed set in a resolution 
    $\phi: \Tld{X} \rightarrow X$ then there exists a function 
    $f \in \Rrat{X}$ such that $F = \Z(f)$. 
\end{proposition}
\begin{proof}
    Suppose $f$ is a regular function on $\Tld{X}$ such that $\Z(f) =
    \phi^{-1}(F)$ and that $U = X \setminus C$ where $C$ is the exceptional 
    locus of $\phi$. Then the function
    $\phi^{-1} \circ f \in \Rrat{X}$ by Theorem \ref{thm:blowupsmain}, and
    $\Z(\phi^{-1} \circ f) = \phi(\Z(f)) = F$.
\end{proof}

\subsection{Properties of zero sets of locally bounded rational functions}
\label{sec:closedsets}
A subset $F \subseteq X$ of an irreducible, non-singular algebraic variety $X
\subseteq R^n$ is called a \emph{locally bounded rational set} if it is the
zero set of a locally bounded rational function on $X$. This section 
verifies that the definition of these sets satisfies various properties 
that one expects of a zero-set. In addition, it explores the topology 
associated with these sets. 
\begin{proposition}
    \label{prop:eucclosed}
    If $f \in \Rrat{X}$ where $X \subseteq R^n$ is an irreducible, non-singular, 
    algebraic variety, then $\Z(f)$ is closed in the euclidean topology 
    and is a semi-algebraic set. 
\end{proposition}
\begin{proof}
    This follows directly from the graph based definition of $\Z(f)$. That is, 
    it is the intersection of two closed semi-algebraic sets: 
    $\overline{\mathcal{G}_f}$ and $\{(x, y) \in X \times R | y = 0 \}$. 
\end{proof}

\begin{proposition}
    \label{prop:zero-set-intersection}
    If $f, g \in \Rrat{X}$, then: 
    \begin{itemize}
        \item[(i)] $\Z(fg) = \Z(f) \cup \Z(g)$. 
        \item[(ii)] $\Z(f^2 + g^2) \subseteq \Z(f) \cap \Z(g)$. 
    \end{itemize}
    If, in addition, either $f$ or $g$ is continuous, then (ii) holds with
    equality.
\end{proposition}

\begin{proof}
    Let $\phi: \Tld{X} \rightarrow X$ be a resolution such that both $\Tld{g} =
    g \circ \phi$ and $\Tld{f} = f \circ \phi$ are regular. Now,
    \begin{eqnarray*}
        \Z(fg) &=& \phi(\Z(\Tld{f}\Tld{g}))\\
        &=& \phi(\Z(\Tld{f}) \cup \Z(\Tld{g})) \\
        &=& \phi(\Z(\Tld{f})) \cup \phi(\Z(\Tld{g})) \\
        &=& \Z(f) \cup \Z(g).
    \end{eqnarray*}
    For (ii), $\Z(f^2 + g^2) \subseteq \Z(f) \cap \Z(g)$ is obvious. If 
    $x\in \Z(f)$ then there exists a resolution 
    $\phi:\Tld{X} \rightarrow X$ 
    such that $\Tld{f} = f \circ \phi$ is regular, 
    and  $\Tld{x} \in \Tld{X}$ such 
    that $\Tld{f}(\Tld{x}) = 0$. Now, 
    if, in addition, $g$ is continuous and $x \in \Z(g)$ then,
    $g(x) = 0$ and $( g \circ \phi )(\Tld{x}) = 0$ as 
    $g \circ \phi$ is zero at all points of 
    $\phi^{-1}(x)$. Therefore, 
    $\Tld{x} \in \Z(\Tld{f}^2 + \Tld{g}^2)$ which 
    implies $x \in \Z(f^2 + g^2)$. 
\end{proof}
The following example shows that, in general, one does not 
have equality for Proposition \ref{prop:zero-set-intersection} (ii).   
\begin{example}
    \label{rem:cex}
    Let $f = x^2/(x^2 + y^2)$ and $g = y^2/(x^2 + y^2)$. Then, 
    $f, g \in \Rrat{R^2}$ and $\Z(f^2 + g^2) = \varnothing$, whereas 
    $\Z(f)$ and $Z(g)$ both contain the origin. 
\end{example}

\begin{proposition}
    \label{prop:productzeroset}
    If $f \in \Rrat{X}$ and $g \in \Rrat{Y}$ where $X$ and $Y$ are two
    irreducible, non-singular, algebraic varieties, then there exists $h \in
    \Rrat{X \times Y}$ such that $\Z(h) = \Z(f) \times \Z(g)$.
\end{proposition}
\begin{proof}
    Let $\pi_X: X \times Y \rightarrow X$ and $\pi_Y: X\times Y \rightarrow Y$
    be the corresponding coordinate projections of $X\times Y$ onto $X$ and $Y$
    respectively.  Note that $\Z(f \circ \pi_X) = \Z(f) \times Y$ and $\Z(g
    \circ \pi_Y) = X \times \Z(g)$ and $\Z(f) \times \Z(g) = (\Z(f) \times Y)
    \cap (X \times \Z(g))$.  Let $h = (f \circ \pi_X)^2 + (g \circ \pi_Y)^2$. By
    Proposition \ref{prop:zero-set-intersection}, $\Z(h) \subseteq \Z(f) \times
    \Z(g)$.

    If $(x, y) \in \Z(f) \times \Z(g)$, and $\alpha: [0, 1] \rightarrow X$ and
    $\beta: [0, 1] \rightarrow Y$, are two continuous, 
    semi-algebraic arcs such that $\lim_{t \to 0} ( f\circ \alpha ) (t) = 0$ 
    and $\lim_{t \to 0} ( g \circ \beta ) (t) = 0$ 
    then, $\gamma = (\alpha, \beta) : [0, 1] \rightarrow X \times Y$ is
    a semi-algebraic, continuous, arc such that 
    $ \lim_{t \to 0}( h \circ \gamma ) (t) = (0, 0)$
    because $\alpha = \gamma \circ \pi_{X}$ and $\beta = \gamma \circ \pi_Y$,
    and hence $\Z(f) \times \Z(g) \subseteq \Z(h)$.
\end{proof}
The following two examples serve to demonstrate the fact that locally 
bounded rational sets can contain line segments or semi-lines in 
lower dimensional subspaces.  
\begin{example}
    \label{ex:segment}
    Let $f \in \Rrat{R^3}$ be the function given by,  
    \begin{equation*}
        f = \left ( z - \frac{x^2}{x^2 + y^2} \right )^2 + x^2 + y^2. 
    \end{equation*}
    Observe that, by Example \ref{ex:protoyptical}, the term, 
    $x^2/(x^2 + y^2)$, takes on all values between $0$ and $1$ 
    when $x = y = 0$, implying that $z - x^2/(x^2 + y^2)$ has 
    the line segment $\{(0, 0, t)| 0 \leq t \leq 1\}$ contained 
    within its zero locus, the term $x^2 + y^2$ ensure that there 
    are no other points in the zero set of $f$. 
    Therefore $\Z(f)$ is exactly the line 
    segment $\{(0, 0, t)| 0 \leq t \leq 1\}$. 
\end{example}
\begin{example}
    \label{ex:semi-line}
    Let $f \in \Rrat{R^3}$ be the function defined by,  
    \begin{equation*}
        f(x, y, z) = \left ( \frac{2z}{1 + z^2} - \frac{x^2}{x^2 + y^2} \right)^2 + x^2 + y^2.
    \end{equation*}
    Now, $f \in \Rrat{R^3}$ as it is the composition of a locally 
    bounded rational function with a regular function. The regular 
    function $2z/(1 + z^2)$ takes on all the values in 
    $[0, 1]$ as $z$ varies from $0$ to $+\infty$. 
    From Example \ref{ex:protoyptical} therefore 
    the entirety of the positive $z$-axis is included in the 
    zero set of the term in parenthesis in the definition of $f$ above. 
    Therefore, $\Z(f) = \{(x, y, z) \in R^3 | x = 0, y = 0, z \geq 0\}$. 
\end{example}
The phenomenon in Example \ref{ex:semi-line} can be generalized 
using Proposition \ref{prop:productzeroset} as follows:
\begin{proposition}
    \label{prop:coordinateplane}
    For any integer $k$, let  $V \subseteq R^{k}$ be 
    the semi-algebraic set given by 
    $\{(y_1, \dots, y_k) \in R^k | y_1 \geq 0 \dots y_k \geq 0\}$, 
    and $(0)_{2k}$ be the origin in $R^{2k}$. Then $V \times (0)_{2k}$ is 
    a locally bounded rational set. 
\end{proposition}
The following proposition shows that any semi-algebraic set is isomorphic 
to a locally bounded rational set embedded in a higher dimensional ambient
euclidean space. This serves to illustrate an interesting phenomenon that 
does not occur for zero sets of 
regulous functions (cf. \cite{Fic16}). 
\begin{proposition}
    \label{prop:poly-proj}
    If $U = \{ x \in R^n | p_1(x) \geq 0, \dots, p_k(x) \geq 0 \}$ is a closed
    semi-algebraic set where $p_i$ are polynomials (for $1 \leq i \leq k$), then
    there exists $h \in R_{b}(R^{n + 3k})$ such that $\Z(h) \cong U$ via
    $\phi:R^{n + 3k} \rightarrow R^n$, the projection onto the first $n$
    coordinates.
\end{proposition}
\begin{proof}
    Let 
    $V = \{(x, y) \in R^n \times R^k | 
    p_i (x) - y_i = 0, y_i \geq 0, \forall i \text{ s.t. }1 \leq i \leq k\}$, 
    where $x = (x_1, \dots, x_n), y = (y_1, \dots, y_k)$.

    Note that $V = U_1 \cap U_2$ where, 
    \begin{eqnarray*}
        U_1 &=& \{ (x, y) \in R^{n + k} | p_i(x) - y_i = 0,\;\; \forall i \text{ s.t. } 1 \leq i \leq k\}\\
        U_2 &=& \{ (x, y) \in R^{n + k} | y_i \geq 0, \;\; \forall i \text{ s.t. }1 \leq i \leq k\}.
    \end{eqnarray*}
    Therefore, 
    \begin{equation*}
        V \times (0)_{2k} = (U_1 \times R^{2k}) \cap (U_2 \times (0)_{2k}). 
    \end{equation*}
    Now $U_1 \times R^{2k} = \Z(h_1)$, where, 
    \begin{equation*}
        h_1(x, y, \tilde{x}) = \sum_{i = 1}^{k} (p_i(x) - y_i)^2, 
    \end{equation*}
    where $\tilde{x} = (\tilde{x}_1, \dots, \tilde{x}_{2k})$. 
    Also, there exists $h_2 \in \Rrat{R^{n+3k}}$ 
    such that $U_2 \times (0)_{2k} = \Z(h_2)$, 
    by Proposition \ref{prop:coordinateplane}. Now, by 
    Proposition \ref{prop:zero-set-intersection},
    \begin{equation*}
        V \times (0)_{2k} = \Z(h), 
    \end{equation*}
    where $h = h_1^2 + h_2^2$. 
    Further, it is easy to verify that 
    $\tilde{\pi} : U \rightarrow V \times (0)_{2k}$, 
    given by, 
    \begin{equation*}
        \tilde{\pi} (x) = (x, p_1(x), \dots, p_k(x), \underbrace{0, \dots, 0)}_{2k \text{ times }} 
    \end{equation*}
    is an inverse of 
    $\pi\vert_{V \times (0)_{2k}} : V \times (0)_{2k} \rightarrow U$, 
    where $\pi$ is the projection onto the first $n$ coordinates, therefore 
    $\pi\vert_{V \times (0)_{2k}}$ is an isomorphism. 
\end{proof}

\begin{theorem}
    \label{thm:polynomialisomorph}
    Every closed semi-algebraic set is polynomially isomorphic (via a
    projection) to a locally bounded rational set.  
\end{theorem}
\begin{proof}
    This follows from \cite[2.7.2]{BCR13} which states that every closed
    semi-algebraic set is a finite union of closed semi-algebraic sets of the
    form $\{x \in R^n | f_1(x) \geq 0, \dots, f_k(x) \geq 0 \}$, and Proposition
    \ref{prop:poly-proj} above.
\end{proof}

\begin{corollary}
    \label{cor:projection}
    Every closed semi-algebraic subset of $R^n$ is the image of a Zariski closed
    set of $R^{n+k}$, for some $k$, by the projection onto the first $n$
    coordinates.
\end{corollary}
\begin{proof}
    By \cite[Proposition 3.5.8]{BCR13} a blow-up can be expressed as a projection. 
    The projection of Theorem \ref{thm:polynomialisomorph} 
    may be composed with the composition of blow-ups from 
    Theorem \ref{thm:blowupsmain}. 
\end{proof}

\begin{proposition}
    \label{prop:topbasic}
    The complements of the euclidean closed sets of the form $\Z(f)$ for $f \in
    \Rrat{X}$ where $X$ is an irreducible, non-singular, algebraic variety, form
    the basis of a topology on $X$.
\end{proposition}
\begin{proof}
    If $f, g \in R_{b}(X)$, then it suffices to prove that $\Z(f) \cup \Z(g)$
    is the zero set of a locally bounded rational function. By Proposition
    \ref{prop:zero-set-intersection}, $\Z(f) \cup \Z(g) = \Z(fg)$.
\end{proof}
\begin{remark}
    \label{rem:constructible}
    Examples \ref{ex:segment} and \ref{ex:semi-line} demonstrate that 
    the topology referred to in Proposition \ref{prop:topbasic} is 
    strictly finer than the Zariski constructible topology on $R^n$.
    That is, where the closed sets are finite intersections and 
    unions of Zariski closed sets. 
    This is not the case for regulous
    functions \cite{Fic16} for which the topology 
    associated to their zero-sets is the same as the 
    Zariski constructible topology. 
\end{remark}
\begin{example}
    \label{ex:noetherian}
    The topology of Proposition \ref{prop:topbasic} above is not 
    Noetherian. Let $\alpha \in R$, $\alpha > 0$. The function 
    \begin{equation*}
        f_{\alpha} (x, y, z) = \left ( z - \alpha \frac{x^2}{x^2 + y^2}\right )^2 +  x^2 + y^2 
    \end{equation*}
    is locally bounded on $R^3$ and has 
    $\Z(f_{\alpha}) = \{(0, 0, t) | 0 \leq t \leq \alpha \}$. 
    Then the collection $\{\Z(f_{1 + 1/n}\}_{n \in \N}$ forms an 
    infinitely decreasing chain of closed sets that does not 
    stabilise. 
\end{example}
\begin{proposition}
    \label{prop:invertible}
    If $f \in \Rrat{X}$ for an irreducible, non-singular algebraic variety $X$
    and $\Z(f) = \varnothing$, then $f$ is invertible in $\Rrat{X}$ (i.e. 
    $\Ang{ f } = \Rrat{X}$).
\end{proposition}
\begin{proof}
    This is a consequence of the fact that in this case $1/f \in \Rrat{X}$.
    This is because there exists no semi-algebraic arc $\gamma:[0, 1]
    \rightarrow X$ such that $\lim_{t \to 0} ( f\circ \gamma ) (t) = 0$, which
    implies that there is no semi-algebraic arc $\gamma$ such that $\lim_{t \to
    0} ( (1/f) \circ \gamma ) (t) = \infty$. Therefore, $f.(1/f) = 1 \in \Ang{ f }$,
    which implies that $\Ang{ f } = R_{b}(X)$.
\end{proof}

\subsection{Images of locally bounded rational maps}
\label{sec:images}

This section explores the geometry of the images of locally bounded rational
maps. As in the case for zero-sets of locally bounded rational functions, 
there are many ways to define these. Three equivalent ways corresponding 
to the three equivalent definitions of the zero-sets presented in 
Section \ref{sec:zero-set} are considered here. 

Let $X, Y$ be irreducible, non-singular, algebraic varieties $f \in \Rrat{X, Y}$
and $U=\domain(f)$.  The image of $f$ can be defined in one of three ways:
\begin{itemize}
    \item[(1)] Via arcs: $\Image_{arc}(f) \coloneqq \{ a \in Y | \exists \gamma:[0, 1] \rightarrow
        X, \text { semi-algebraic with } \gamma((0,1]) \subseteq U \text{ such that }\lim_{t\to0}(f(\gamma(t)) = a \}$. 
    \item[(2)] Via the graph: $\Image_{graph}(f) \coloneqq \{ a \in Y | \exists x \in X \text{ such that } (x, a) \in \overline{\graph{f}}\}$
    \item[(3)] Via resolutions: $\Image_{res}(f) = \Image (f \circ \phi)$, where 
        $\phi: \Tld{X} \rightarrow X$ is a composition of a finite number of 
        blowings up in smooth centres such that $f \circ \phi$ is regular. 
\end{itemize}
\begin{proposition}
    \label{prop:imageequiv}
    The three sets defined above are the same, 
    i.e. $\Image_{arc}(f) = \Image_{graph}(f) = \Image_{res}(f)$. 
\end{proposition}
\begin{proof}
    $\Image_{arc}(f) \subseteq \Image_{graph}(f)$: \\
    Let $\gamma: [0, 1] \rightarrow X$ be a semi-algebraic arc with
    $\gamma((0,1]) \subseteq \domain(f)$, $a = \lim_{t \to 0} ( f \circ \gamma )
    (t)$, and $x = \gamma(0)$. Observe that $\eta(t) = (\gamma(t), f(\gamma(t)))$
    is a semi-algebraic arc inside $X \times Y$ and $\eta((0, 1]) \subseteq
    \graph{f}$. Therefore, $\lim_{t \to 0} \eta (t) = (x, a) \in
    \overline{\graph{f}}$ and $\Image_{arc}(f) \subseteq
    \Image_{graph}(f)$.

    $\Image_{graph}(f) \subseteq \Image_{arc}(f)$:\\
    Now, let $(x, a) \subseteq \overline{\graph{f}}$. By the curve
    selection lemma \cite[2.5.5]{BCR13}, there exists a semi-algebraic arc
    $(\alpha, \beta): [0, 1] \rightarrow X \times Y$, with $(\alpha, \beta)((0,
    1]) \subseteq \mathcal{G}_{f}$ such that $\lim_{t \to 0} (\alpha(t),
    \beta(t)) = (x, a)$. Now, by the definition of $\graph{f}$, for $t
    \neq 0$, $f(\alpha(t)) = \beta(t)$ and $\lim_{t \to 0} f(\alpha(t)) = a$.
    Therefore, $\Image_{graph}(f) \subseteq \Image_{arc}(f)$.

    $\Image_{arc}(f) \subseteq \Image_{res}(f)$:\\
    Let again $\gamma$ be a semi-algebraic arc,
    and $a = \lim_{t \to 0} ( f \circ \gamma ) (t)$. If $\phi:\Tld{X}
    \rightarrow X$ is a resolution that makes $\Tld{f} = f \circ \phi$ regular,
    then $\Tld{\gamma} = \phi^{-1} \circ \gamma : (0, 1] \rightarrow \Tld{X}$
    extended to zero by continuity is a semi-algebraic arc. If $\Tld{x} =
    \lim_{t \to 0} \Tld{\gamma}(t)$ then on $(0, 1]$,
    \begin{equation*}
        f \circ \gamma = f \circ \phi \circ \phi^{-1} \circ \gamma  = \Tld{f} \circ \Tld{\gamma}. 
    \end{equation*}
    Therefore, $a = \lim_{t \to 0} ( f \circ \gamma ) (t) = \lim_{t \to 0} ( \Tld{f}
    \circ \Tld{\gamma}) (t) $, and $\Tld{f} (\Tld{x}) = a$ by the continuity of
    $\Tld{f}$. Therefore $\Image_{arc} (f) \subseteq \Image_{res} (f)$.

    $\Image_{res}(f) \subseteq \Image_{arc}(f)$:\\
    Conversely, if $\Tld{f}(\Tld{x}) = a$ for some $\Tld{x} \in \Tld{X}$, then
    there exists a semi-algebraic arc $\Tld{\gamma}:[0, 1] \rightarrow \Tld{X}$
    with $\Tld{\gamma}((0, 1]) \subseteq \phi^{-1}(\domain(f))$ such that
    $\lim_{t \to 0} \Tld{\gamma}(t) = \Tld{x}$, by the curve selection lemma
    (Theorem \ref{thm:curveselection}). Now, if $\gamma = \phi \circ \Tld{\gamma}$, then
    $\gamma((0, 1]) \subseteq \domain (f)$ and $a = \lim_{t \to 0} ( \Tld{f} \circ
    \Tld{\gamma} ) (t) = \lim_{t \to 0} ( f \circ \gamma )(t)$.  Therefore
    $\Image_{res} (f) \subseteq \Image_{arc} (f)$.
\end{proof} 
As a consequence of the above, the notation $\Image(f)$ will be used for all of
$\Image_{arc}(f)$, $\Image_{graph}(f)$, and $\Image_{res}(f)$ in what follows.
The following lemma is an immediate consequence of the previous proposition and
the definition of $\Z(f)$.
\begin{lemma}
    \label{lem:zero}
    If $f \in R_{b}(X)$ where $X$ is an irreducible, non-singular algebraic
    variety, then $\Z(f) = \varnothing$ if and only if $0 \notin \Image(f)$.
\end{lemma}
\begin{proposition}
    \label{prop:connected}
    If $f \in \R_{b}(X, Y)$, where $X$, $Y$ are irreducible, non-singular,
    algebraic varieties, then $\Image(f)$ is semi-algebraic, closed and bounded
    if $X$ is so. It is also semi-algebraically connected if $X$ is so.
\end{proposition}
\begin{proof}
    From the definition of $\Image(f)$ via graphs, it is semi-algebraic as 
    it is the projection of a semi-algebraic set (see \cite[2.2.1]{BCR13}). 
    The rest follows from the Theorem \ref{thm:blowupsmain} and the 
    fact that the resolution map is finite and proper. 
\end{proof}
The following theorem is an immediate consequence of Theorem
\ref{thm:blowupsmain}. 
\begin{theorem}
    \label{thm:imagereg}
    A set is an image of a locally bounded rational function if and only if it is an 
    image of a regular function.  
\end{theorem}
\begin{proposition}
    \label{prop:indetfibres}
    If $f \in R_{b}(X, Y)$ and $x \in \indet(f)$, then the 
    set $f(\{x\}) \coloneqq 
    \{y \in Y | \exists \text{ a semialgebraic arc }\gamma:[0, 1] \rightarrow X \text{ with } \lim_{t\to 0} \gamma(t) = x \text { and } \lim_{t\to 0}(f\circ\gamma ) (t) = y\}$
    is semi-algebraically closed and connected. 
\end{proposition}
\begin{proof}
    If $\phi:\Tld{X} \rightarrow X$ is a resolution that makes $\Tld{f} = f
    \circ \phi$ regular, then $\phi^{-1}(x)$ is closed and bounded because
    $\phi$ is a proper map. It is also semi-algebraic as the inverse image of a
    semi-algebraic set by a semi-algebraic map.  Now, since $f$ is regular, and
    hence continuous, by \cite[Theorem 2.5.8]{BCR13}, 
    $\Tld{f}(\phi^{-1}(x)) = f(\{x\})$
    is a closed and bounded semi-algebraic set.
\end{proof}

\subsection{A \L{}ojasiewicz inequality for locally bounded rational functions}
\label{sec:linequality}
In this section versions of \L{}ojasiewicz's inequality for 
locally bounded rational functions will be established. 

\begin{theorem}
    \label{thm:lineqmain}
    Let $f, g \in \Rrat{X}$, where $X \subseteq R^n$ is an irreducible
    non-singular, algebraic variety. If  $\phi : \Tld{X} \rightarrow X$
    composition of blowups with smooth centres 
    such that $\indet(f\circ \phi) = \indet(g\circ \phi) = \varnothing$ and
    $\Z(g\circ \phi) \subseteq \Z(f \circ \phi)$, then there exists and an 
    integer $N$ such that $f^N/g \in \Rrat{X}$.
\end{theorem}
\begin{proof}
    Let $U = \domain(f)$, $V = \domain(g)$, $\Tld{U} = \phi^{-1}(U)$, $\Tld{V} =
    \phi^{-1}(V)$, $W = U \cap V$ and $\Tld{W} = \phi^{-1}(W)= \Tld{U} \cap
    \Tld{V}$.  Further, let $\Tld{f} = f \circ \phi$, $\Tld{g} = g \circ \phi$,
    $Z = \{x \in W | g(x) \neq 0\}$ and $\Tld{Z} = \phi^{-1}(Z)$. Since $W$ is
    an intersection of two dense Zariski open sets, it is a dense Zariski open
    set and $f$ and $g$ can be considered as functions defined on $W$.

    By the hypotheses, $\Tld{g}$ is non-zero on $\{x \in
    \Tld{X}|\Tld{f}(x) \neq 0 \}$, therefore $1/\Tld{g}$ is continuous on this
    set. It is also semi-algebraic as $g$ is semi-algebraic.  By the
    \L{}ojasiewicz inequality for semi-algebraic functions \cite[2.6.4]{BCR13}
    there exists an integer $N$ such that $\Tld{f}^{N}/\Tld{g}$ extended by zero
    on $\Z(\Tld{f})$ is continuous on the whole of $\Tld{X}$.

    Now on $\Tld{Z}$,
    \begin{equation*}
        \frac{\Tld{f}^N}{\Tld{g}} = \frac{(f\circ \phi)^N}{g \circ \phi} = \left ( \frac{f^N}{g} \right ) \circ \phi
    \end{equation*}
    If $K$ is a closed and bounded subset of $X$, then $\Tld{K} = \phi^{-1}(K)$
    is a closed and bounded subset of $\Tld{X}$ as the map $\phi$ is proper. As
    $\Tld{f}^{N}/\Tld{g}$ is continuous on $\Tld{K}$ it is also bounded
    \cite[2.5.8]{BCR13}. In particular, $\frac{\Tld{f}^{N}}{\Tld{g}}(\Tld{W}
    \cap \Tld{K})$ is a bounded set.  Now, as $Z \subseteq W$ and
    $\frac{\Tld{f}^{N}}{\Tld{g}} (\Tld{Z} \cap \Tld{K}) = \frac{f^N}{g}(Z \cap
    K)$, this last set is also bounded. 
    Then, by Proposition \ref{prop:compact} the function $\frac{f^N}{g}$ 
    is a locally bounded rational function. 
\end{proof}
The following result permits the formulation of a 
corresponding \L{}ojasiewicz-type inequality 
result in terms of arcs. 
\begin{proposition}
    \label{prop:helperlineqarc}
    If $f$, $g \in \Rrat{X}$ where $X$ is an irreducible, non-singular, 
    algebraic variety, then the following statements are equivalent: 
    \begin{itemize}
        \item[(i)] There exists a resolution $\phi: \Tld{X} \rightarrow X$ such that 
            $\indet(f\circ \phi) = \indet(g \circ \phi) = \varnothing$ 
            and $\Z(f \circ \phi) \subseteq \Z(g \circ \phi)$.  
        \item[(ii)] For every continuous, semi-algebraic 
            arc $\gamma: [0, 1] \rightarrow X$ such that
            $\gamma((0,1]) \subseteq \domain(f) \cap \domain(g)$ the following holds, 
            \begin{equation}
                \lim_{t\to 0}f(\gamma(t)) = 0 \implies \lim_{t \to 0}g(\gamma(t)) = 0. 
            \end{equation}
        \item[(iii)] For every resolution $\phi:\Tld{X} \rightarrow X$ such that 
            $\indet(f\circ \phi) = \indet(g \circ \phi) = \varnothing$, one has 
            $\Z(f \circ \phi) \subseteq \Z(g \circ \phi)$. 
    \end{itemize}
\end{proposition}
\begin{proof}
    (iii) $\implies$ (i): This follows directly from the fact that there exists 
    a common resolution that renders two locally bounded rational functions 
    regular. Note here that the resolution is an isomorphism on the intersection
    $\domain(f) \cap \domain(g)$. 

    (i) $\implies$ (ii): Suppose that $\gamma$ is a semi-algebraic such that
    $\gamma((0, 1]) \subseteq \domain(f) \cap \domain(g)$ and 
    $\lim_{t \to 0} f(\gamma(t)) = 0$. 
    The resolution $\phi:\Tld{X} \rightarrow X$ is an isomorphism between 
    $\phi^{-1} (\domain(f)) \cap \phi^{-1}(\domain(g))$ and 
    $\domain(f) \cap \domain(g)$, therefore 
    the arc $\gamma$ can be lifted through the resolution $\phi$ on 
    all points of its domain other than $0$ to $\Tld{\gamma}$ 
    with 
    $\Tld{\gamma}((0, 1]) \subseteq \phi^{-1}(\domain(f)) \cap \phi^{-1} (\domain(g))$. 
    This arc, $\Tld{\gamma}$, can be extended by continuity to $0$ obtaining
    $\Tld{f}(\Tld{\gamma}(0)) = 0$, 
    where $\Tld{f}$ is $f \circ \phi$. Now letting $\Tld{g} = g \circ \phi$, 
    this  implies, by 
    the hypothesis, that $\Tld{g}(\Tld{\gamma}(0)) = 0$, 
    which, in turn, implies that $\lim_{t \to 0} g(\gamma(t)) = 0$. 

    (ii) $\implies$ (iii): Let $U = \domain(f)$, $V = \domain(g)$, 
    $\Tld{U} = \phi^{-1}(U)$ and $\Tld{V} = \phi^{-1}(V)$, for 
    a resolution satisfying the hypotheses of (iii). Further 
    suppose that $\Tld{x} \in \Z(\Tld{f})$. Then 
    $\Tld{U}$ and $\Tld{V}$ are dense Zariski open subsets 
    of $\Tld{X} = \phi^{-1}(X)$, and by the curve 
    selection lemma \cite[2.5.5]{BCR13}, there exists a
    semi-algebraic arc $\Tld{\gamma}: [0, 1] \rightarrow \Tld{X}$ such 
    that $\Tld{\gamma}((0, 1])$ and $\lim_{t \to 0}(\Tld{\gamma}(t)) = \Tld{x}$. 
    Let $\gamma = \phi \circ \Tld{\gamma}$. This is a semi-algebraic arc and 
    it is easy to see that $\lim_{t \to 0} f(\gamma(t)) = \Tld{f}(\Tld{x})$. 
    By (ii) this implies that $\lim_{t \to 0} g(\gamma(t)) = 0$, which, 
    in turn, implies that  $\Tld{g}(\Tld{x}) = 0$. 
    Therefore $\Tld{x} \in \Z(\Tld{g})$. 
\end{proof}
As stated previously, 
Proposition \ref{prop:helperlineqarc} allows the formulation of 
Theorem \ref{thm:lineqmain} in terms of arcs as follows:
\begin{theorem}
    \label{thm:lineqarc}
    If $f$ and $g \in \Rrat{X}$ where $X \subseteq R^n$ is an irreducible,
    non-singular, algebraic variety and for every
    semi-algebraic arc $\gamma:[0, 1] \rightarrow X$, such that $\gamma((0,1])
    \subseteq \domain(f) \cap \domain(g)$ the following holds,
    \begin{equation}
        \lim_{t\to 0}f(\gamma(t)) = 0 \implies \lim_{t \to 0}g(\gamma(t)) = 0, 
    \end{equation}
    then there exists an integer $N$ such that $f^N/g \in\Rrat{X}$.  
\end{theorem}

\begin{remark}
    It should be noted here that the entirety of the work in this 
    article up to this point can be done exclusively using resolutions 
    of singularities without making any reference to arcs. That is, 
    by defining "values" of a locally bounded rational functions on 
    their loci of indeterminacy by taking the values of 
    their associated regular functions after resolution at points 
    inside the fibres over the loci of indeterminacy. 
\end{remark}

\section{Zeros in arc spaces}

The motivation of this section is to explore the possibility of 
developing an algebro-geometric dictionary that enables one to 
relate ideals in the ring of locally bounded rational functions and 
geometric sets defined by them, similar to what exists for the 
class of regulous functions (see, for example, \cite{Fic16}). 
The definition of the zero-sets of locally bounded rational functions 
given in Section \ref{sec:zero-set} does not extend to the definition 
of the zero-set of a collection of locally bounded rational functions 
(see Example \ref{ex:arc-spaceex} below). One way to overcome this 
problem is to consider zero-sets in the arc space of the 
irreducible, non-singular algebraic variety on which the 
locally bounded rational functions are defined. This section explores this 
approach, and succeeds, in the case where the domain has dimension 2, 
in reconstructing an algebro-geometric dictionary by utilising this 
approach. 

\subsection{Zeros in arc spaces}
\label{sec:arcspaces}
\begin{example}
    \label{ex:arc-spaceex}
    Let $X$ be a non-singular, irreducible algebraic variety, and 
    $A$ be a subset of $\Rrat{X}$. Further let,   
    \begin{equation*}
        \Z_1(A) = \{ x \in X | \forall f \in A, \exists \gamma \text{ a semi-algebraic arc in }X, \text{ s.t. } \lim_{t \to 0} (f \circ \gamma)(t) = 0\}
    \end{equation*}
    Now taking $X = R^2$, consider the functions,
    \begin{equation*}
        f = \frac{x^2 + y^4}{x^2 + y^2}, 
    \end{equation*}
    \begin{equation*}
        g = \frac{x^4 + y^2}{x^2 + y^2}. 
    \end{equation*}
    Then it is clear that $f, g \in \Rrat{R^2}$, and
    that $(0, 0) \in \Z(f) \cap \Z(g)$. However,  
    $f + g = 1 + (x^4 + y^4)/(x^2 + y^2)$, which implies that, 
    $\varnothing = \Z_1(\{f, g, f + g\}) \supseteq \Z_1 (\Ang{f, g})$.
    Indicating that some care is called for when defining the zero-set of 
    a collection of locally bounded rational functions. 
\end{example}

Let $X$ be a non-singular, irreducible, algebraic variety, and 
let $\Arc{X}$ be its arc-space. That is the space of 
germs of non-constant semi-algebraic arcs associated with $X$ 
(as defined at the end of Section \ref{sec:background}). 
If $f \in \Rrat{X}$, then, 
\begin{equation*}
    \Arc{\Z}(f) \coloneqq \{\alpha \in \Arc{X} | \exists \epsilon > 0 \text{ such that } \alpha((0, \epsilon]) \cap \indet(f) = \varnothing, \lim_{t\to 0} f(\alpha (t)) = 0\}, 
\end{equation*}
will be called \emph{the set of zeros of $f$ in $\Arc{X}$}.  
This is the set of germs of non-constant, continuous semi-algebraic arcs in 
$X$ which do not intersect with the locus of indeterminacy of $f$ such that, 
$f$ tends to zero along them as $t$ tends to $0$. 
Utilizing this new definition it is possible to 
restate Theorem \ref{thm:lineqarc} in a more classical form for 
a \L{}ojasiewicz-type inequality result (similar, for example, to the 
one in \cite{Fic16} for regulous functions.) 
\begin{theorem}
    \label{thm:lineqarc2}
    If $f, g \in \Rrat{X}$, where $X$ is a non-singular, irreducible, algebraic 
    variety, and if $\Arc{\Z}(g) \subseteq \Arc{\Z}(f)$, then 
    there exists an integer $N$ such that $f^N/g \in \Rrat{X}$.
\end{theorem}

Let now $A$ be a subset of $\Rrat{X}$. Then the set  
$\Arc{\Z}(A) \coloneqq \bigcap_{f \in A} \Arc{\Z}(f)$, will
be called the \emph{set of common zeros of $A$ in $\Arc{X}$}.

As the following two results will show, when $\dim X \geq 3$, 
even the new definition of zeros in arc-spaces does not yield 
a useful notion of the zeros associated to an ideal of $\Rrat{X}$.  
This is a direct consequence of the fact that if $f \in \Rrat{X}$, 
then $\dim ( \indet (f) )$ may be greater than $1$ is 
$\dim X \geq 3$. 
\begin{lemma}
    \label{prop:nonset}
    Suppose $f \in \Rrat{R^n}$, for $n \geq 3$. 
    Further, suppose that $\gamma:[0, 1] \rightarrow \Z(f)$ is a 
    semi-algebraic arc and that there exist $\epsilon_1$ and $\delta_1$, 
    satisfying $0 < \epsilon_1 < \delta_1 < 1$, for which 
    $\gamma([\epsilon_1, \delta_1])$ is non-singular. Then there exists a 
    function $g_{\gamma} \in  \Rrat{X}$, such that 
    for some non-zero $\epsilon \leq \epsilon_1$ and $\delta$ satisfying 
    $\epsilon \leq \delta \leq \delta_1$, 
    $\gamma([\epsilon, \delta]) \subseteq \indet(g_{\gamma})$. 
\end{lemma}
\begin{proof}
  By \cite[Proposition 3.3.10]{BCR13} there 
  exist $n-1$ polynomials $P_1, \dots, P_{n-1}$ 
  on $R^n$ such that for some $\epsilon, \delta$ satisfying, 
  $0 < \epsilon_1 \leq \epsilon < \delta \leq \delta_1 < 1$, 
  \begin{equation*}
    \gamma([\epsilon, \delta]) \subseteq \Z(P_1, \dots, P_{n-1}). 
  \end{equation*}
  Then the function 
  \begin{equation*}
    g_{\gamma} = \frac{P_1^2}{P_1^2 + \dots + P_{n-1}^2}, 
  \end{equation*}
  satisfies the required property. 
\end{proof}
The above Lemma implies that if $\Arc{\Z}(\Ang{ f }) = \bigcap_{g \in \Ang{ f }} \Arc{\Z}(g)$,
then every arc in $\Arc{R^n}$ would be excluded from some set on
the right hand side, yielding the following:
\begin{proposition}
    \label{thm:nonutil}
    If $n \geq 3$, $f \in \Rrat{R^n}$, and  
    $\Ang{ f } \subseteq \Rrat{R^n}$ is the 
    ideal generated by $f$, then 
    $\Arc{\Z}(\Ang{ f }) = \varnothing$. 
\end{proposition}

The next section will show that for $n \leq 2$, however, 
$\Arc{\Z}(I)$ for $I$ an ideal in the ring of locally 
bounded rational functions is a non-trivial concept that 
has utility. 

\subsection{The case of dimension 2}
\label{sec:dimension2}

When $\dim X = 2$, Theorem \ref{thm:codim} implies that, 
$\dim ( \indet (f) )= 0$ for every $f \in \Rrat{X}$, therefore 
it consists of isolated points. These correspond to 
germs of constant arcs which are explicitly removed 
in the definition of $\Arc{X}$ in Section \ref{sec:background}.  
As a result of this, an arc $\alpha \in \Arc{X}$ will always 
satisfy $\alpha ((0, \epsilon]) \cap \indet (f) = \varnothing$, 
for every $f \in \Rrat{X}$. Consequently, as this section will 
demonstrate, it is possible in this case, to establish, for 
locally bounded rational functions, an algebro-geometric 
dictionary between ideals and zero-sets similar to the 
one that exists 
for other classes of functions in real algebraic geometry 
such as polynomials and regulous 
functions (cf. \cite{Fic16}).

If $\Lambda$ is a subset of $\Arc{R^2}$ then the 
\emph{annulator ideal of $\Lambda$} is defined as 
$\Arc{\mathcal{I}}(\Lambda) = \{f \in \Rrat{R^2} | \forall \alpha \in \Lambda, 
\lim_{t\to 0} (f\circ \alpha) (t) = 0\}$. In order to justify this 
terminology it is necessary to establish that $\Arc{\mathcal{I}} (\Lambda)$ 
is indeed an ideal of $\Rrat{I}$. The following two results accomplish this.  

\begin{lemma}
    \label{lem:valuation1}
    If $f \in \Rrat{R^2}$ then for each $\gamma \in \Arc{R^2}$ there exists 
    $\epsilon > 0$ such that $f \circ \gamma$ is defined and bounded on 
    $(0, \epsilon]$, that is, $f\circ\gamma \in R\Ang{T}_b$.  
\end{lemma}
\begin{proof}
    By Theorem \ref{thm:codim} $\indet(f)$ is a finite set of points. Therefore
    for each non-constant arc $\gamma: [0, 1] \rightarrow R^2$ there exists 
    $\epsilon > 0$ such that $\gamma((0, \epsilon)) \subseteq \domain(f)$, 
    also, $f\circ \gamma$ is bounded by Proposition \ref{prop:arcs}. 
    These together imply that $f \circ \gamma \in R\Ang{T}_b$ by 
    Proposition \ref{prop:puissemialg}. 
\end{proof}
The following is a straightforward consequence of the fact that 
constant arcs have been excluded in the definition of 
$\Arc{R^2}$ and that for all $f \in \Rrat{R^2}$, $\codim (\indet(f)) \geq 2$ 
(Theorem \ref{thm:codim})
\begin{theorem}
    \label{thm:annideal}
    For $\Lambda \subseteq \Arc{R^2}$, the set 
    $\Arc{\I}(\Lambda)$ is an ideal of 
    $\Rrat{R^2}$. 
\end{theorem}
\begin{proof}
    Let $f$ and $g$ be two elements of $\Arc{I}(\Lambda)$. Then by 
    definition $\lim_{t \to 0} (f\circ \alpha) (t) = 0$ and 
    $\lim_{t \to 0} (g\circ \alpha) (t) = 0$ for every 
    $\alpha \in \Lambda$. 
    As $(f \circ \alpha) + (g \circ \alpha) =  ((f + g) \circ \alpha)$ for 
    every $\alpha \in \Lambda$, 
    the limit $\lim_{t \to 0} ((f + g) \circ \alpha) (t)$ 
    is $0$ for every $\alpha \in \Lambda$, 
    implying that $f + g \in \Arc{I}(\Lambda)$. 

    Now, let $g \in \Rrat{R^2}$, and $f \in \Arc{I}( \Lambda )$. Then, 
    by Lemma \ref{lem:valuation1}, $\lim_{t \to 0} (g \circ \alpha) (t)$ exists
    and is finite. Therefore 
    \begin{eqnarray*}
        \lim_{t \to 0} ((g\cdot f) \circ \alpha) (t)  &=& \lim_{t \to 0} \left ( (g \circ \alpha) (t) \cdot (f \circ \alpha) (t) \right ) \\ 
        &=& \lim_{t \to 0} (g \circ \alpha) (t) \cdot 
        \lim_{t \to 0} (f \circ \alpha) (t)\\
        &=& 0.
    \end{eqnarray*}
    This implies that $g \cdot f \in \Arc{I}(\Lambda)$. 
    Therefore $\Arc{I}(\Lambda)$ is an ideal of $\Rrat{R^2}$. 
\end{proof}
The following result verifies that $\Arc{\Z}(\cdot)$ (in dimension 2)  and 
$\Arc{I}(\cdot)$ behave in an expected manner. 
\begin{proposition}
    \label{prop:idealbasicprop}
    \leavevmode \\
    \begin{itemize}
        \item[(i)] For all ideals $I, J \subseteq \Rrat{R^2}$ $I \subseteq J$ 
            implies that $\Arc{\Z}(I) \supseteq \Arc{\Z}(J)$. 
        \item[(ii)] For all $\Lambda_1, \Lambda_2 \subseteq \Arc{R^2}$, 
            $\Lambda_1 \subseteq \Lambda_2$ implies that 
            $\Arc{\I}(\Lambda_1) \supseteq \Arc{\I}(\Lambda_2)$. 
        \item[(iii)] For all $f \in \Rrat{R^2}$, $\Arc{\Z}(f) = \Arc{\Z}(\Ang{f})$. 
    \end{itemize}
\end{proposition}
\begin{proof}
    (i) and (ii) are straightforward. For (iii),  
    if $h \in \Rrat{R^2}$ then, 
    for any $\gamma \in \Arc{\Z}(f)$, $h \circ \gamma$ is 
    bounded by Proposition \ref{prop:arcs} and hence, 
    \begin{eqnarray*}
        \lim_{t\to 0} ((hf)\circ \gamma )(t) &=& \lim_{t\to 0} (h \circ \gamma) (t) \cdot 
                                                \lim_{t \to 0} (f \circ \gamma) (t)\\ 
        &=& 0,   
    \end{eqnarray*}
    as a consequence of the fact that $\lim_{t \to 0} (h \circ \gamma) (t) < \infty$. 
    This implies that $\Arc{\Z}(f) \subseteq \Arc{\Z}(\Ang{f})$.  
\end{proof}
\begin{remark}
    Note here that Proposition \ref{prop:idealbasicprop} (iii) is not true 
    in dimensions greater than or equal to $3$, as was established in 
    Proposition \ref{thm:nonutil}. 
\end{remark}

The following result shows that 
zero-set of a finite number of functions is the same as the 
zero-set of the ideal generated by them. 
\begin{proposition}
    \label{prop:finiteideal2}
    Let $f_1, \dots, f_k \in \Rrat{R^2}$. Then 
    $\Arc{\Z}(\{f_1, \dots, f_k\}) = \Arc{\Z}(\Ang{ f_1, \dots, f_k })$.
\end{proposition}
\begin{proof}
    The inclusion 
    $\Arc{\Z}(\Ang{ f_1, \dots, f_k }) \subseteq \Arc{\Z}(\{f_1, \dots, f_k\})$, 
    follows from the definition of $\Arc{\Z}$. 
    Now, let $\alpha \in \Arc{\Z}(\{f_1, \dots, f_k\})$. If $h \in \Ang{f_1, \dots, f_k}$.
    Then there exist $g_i \in \Rrat{R^2}$, such that $h = \sum_{i = 1}^{k} g_i f_i$ and, 
    \begin{eqnarray*}
        \lim_{t \to 0} ( h \circ \alpha ) (t) &=& \sum_{i = 1}^k\lim_{t \to 0} ( (g_i f_i) \circ \alpha ) (t) \\ 
        &=& \sum_{i=1}^k\lim_{t \to 0} (g_i \circ \alpha) (t) \cdot (f_i \circ \alpha) (t) \\
        &=& \sum_{i=1}^k (\lim_{t \to 0} (g_i \circ \alpha)(t)) ( \lim_{t \to 0} (f_i \circ \alpha)(t)) \\ 
        &=& 0.  
    \end{eqnarray*}
    Where the last equality follows from Lemma \ref{lem:valuation1}, and the 
    fact that $\lim_{t \to 0} (f_i \circ \alpha) (t) = 0$ for all $i$ such 
    that $1 \leq i \leq k$. This implies that 
    $\alpha \in \Arc{ \Z}(\Ang{f_1, \dots, f_k})$
\end{proof}
The following result is a version of the weak Nullstellensatz for
finitely generated ideals in $\Rrat{R^2}$. 
\begin{proposition}
    \label{prop:dim2null}
    Let $f_1, \dots, f_k \in \Rrat{R^2}$. 
    If $\Arc{\Z}(\Ang{ f_1, \dots, f_k }) = \varnothing$, then 
    $\Ang{ f_1, \dots, f_k } = \Rrat{R^2}$. 
\end{proposition}
\begin{proof}
    By Proposition \ref{prop:finiteideal2}, 
    $\Arc{\Z}(\{f_1, \dots, f_k\}) = \Arc{\Z}(\Ang{ f_1, \dots, f_k })$, 
    so the result will be established using the former set. 
    By Corollary \ref{cor:hironaka} there exists a resolution 
    $\phi:\Tld{R^2} \rightarrow R^2$ such that $\Tld{f}_i \coloneqq f_i \circ \phi$ is 
    regular for each $0 \leq i \leq k$. 

    Now, by Theorem \ref{thm:zero-set}, 
    the condition $\Arc{\Z}(\{f_1, \dots, f_k\}) = 
    \bigcap_{0\leq i \leq k} \Arc{\Z}(f_i) = \varnothing$, implies that 
    $\bigcap_{ 0 \leq i \leq k} \Z(\Tld{f}_i) = \varnothing$. By 
    the real Nullstellensatz (cf. \cite[Theorem 4.4.6]{BCR13}), there 
    exist $g_1, \dots g_p \in \Reg{\Tld{R^2}}$ such that 
    $g \coloneqq 1 + \sum_{i = 1}^{p} g^2_p \in \Ang{\Tld{f}_1, \dots, \Tld{f}_k}$. 
    However, $g$ is regular and hence $g^{-1} \in \Reg{\Tld{R^2}}$, which 
    implies that $g \cdot g^{-1} = 1 \in \Ang{\Tld{f}_1, \dots, \Tld{f}_k}$. 

    Therefore, there exist $\Tld{a}_i \in \Reg{\Tld{R^2}}$ such that, 
    \begin{equation*}
        1 = \Tld{a}_1\Tld{f}_1 + \dots + \Tld{a}_k \Tld{f}_k.
    \end{equation*}
    By Theorem \ref{thm:biratiso}, $a_i \coloneqq \Tld{a}_i \circ \phi^{-1} \in \Rrat{R^2}$ 
    for each $0 \leq i \leq k$, which implies that, 
    \begin{equation*}
        1 = a_1 f_1 + \dots + a_k f_k, 
    \end{equation*}
    which, in turn, implies that $\Ang{f_1, \dots, f_k} = \Rrat{R^2}$. 
\end{proof}
Every finitely generated ideal in the ring $\Rrat{R^2}$, has 
the same zero set as a principal ideal. 
\begin{lemma}
    \label{lem:finiteideal1}
    Let $I = \Ang{ f_1, \dots, f_k } \subseteq \Rrat{R^2}$. If 
    $f = f_1^2 + \dots + f_k^2$ then $\Arc{\Z}(f) = \Arc{\Z}(I)$. 
\end{lemma}
\begin{proof}
  This follows from the fact that if $\gamma \in \Arc{R^2}$ then,
  $\lim_{t \to 0} (f \circ \gamma) (t) = 0$ if and only if
  $\lim_{t \to 0} (f_i \circ \gamma) (t) = 0$ for every
  $i \in \{1, \dots, k\}$. 
\end{proof}
The following is a version of the (strong) Nullstellensatz for locally bounded 
rational functions that holds in dimension 2. 
\begin{theorem}
    \label{thm:nullsz1}
    If $I$ is a finitely generated ideal in $\Rrat{R^2}$, then 
    $\Arc{\I}(\Arc{\Z}(I)) = \sqrt{I}$. 
\end{theorem}
\begin{proof}
    Let $f \in \sqrt{I}$, then there exists $n\in \N$ such that $f^n \in I$. 
    If $\gamma \in \Arc{\Z}(I)$ be an arbitrary arc, 
    $\lim_{t \to 0} (f^n \circ \gamma) (t) = 0$. 
    But $f^n \circ \gamma =( f \circ \gamma )^n$. Now, 
    since $f\circ \gamma$ is bounded, and in fact, continuous, 
    as a consequence of Corollary \ref{cor:mapscomp} and 
    Corollary \ref{cor:dim1}, its limit as $t \rightarrow 0$ exists, and 
    therefore, $\lim_{t \to 0} (f \circ \gamma) (t) = 0$, which implies
    that $f \in \Arc{\I}(\Arc{\Z}(I))$. 

    Now, let $f \in \Arc{\I}(\Arc{\Z}(I))$. This implies that 
    $\Arc{\Z}(I) \subseteq \Arc{\Z}(f)$. As $I$ is a finitely 
    generated ideal by Lemma \ref{lem:finiteideal1}, there 
    exist $g_1, \dots, g_k \in I$ such that $\Arc{\Z}(g) = \Arc{\Z}(I)$, 
    for $g \coloneqq \sum_{i = 1}^k g_i^2$. Therefore, 
    $\Arc{\Z}(g) \subseteq \Arc{\Z}(f)$ and by the \L{}ojasiewicz
    inequality (Theorem \ref{thm:lineqarc2}) applied to $f$ and $g$,
    there exists $N \in \N$ such that, $h \coloneqq f^N/g \in \Rrat{R^2}$. 
    This implies that $f^N = gh \in I$ which, in turn, implies that 
    $f \in \sqrt{I}$. 
\end{proof}
\begin{proposition}
    \label{prop:radicals}
    Let $f, g \in \Rrat{R^2}$
    Then $f \in \sqrt{\Ang{ g }}$ if and 
    only if $\Arc{\Z}(g) \subseteq \Arc{\Z}(f)$.
\end{proposition}
\begin{proof}
    Let $f \in \sqrt{\Ang{ g }}$, and $\alpha \in \Arc{\Z}(g)$. 
    By the hypothesis there exist $h \in \Rrat{R^2}$ and $N \in \N$ such that, 
    $f^N = gh$. Therefore, 
    \begin{eqnarray*}
        (\lim_{t \to 0} ( f \circ \alpha ) (t))^N &=& \lim_{t \to 0} ( f^N \circ \alpha ) (t) \\ 
        &=& \lim_{t \to 0} ((gh) \circ \alpha) (t) \\ 
        &=& (\lim_{t \to 0} (g \circ \alpha)(t)) ( \lim_{t \to 0} (h \circ \alpha)(t)) \\ 
        &=& 0, 
    \end{eqnarray*}
    where the last equality follows 
    from the fact that $\lim_{t \to 0} (h \circ \alpha) (t)$ is 
    finite (by Lemma \ref{lem:valuation1}) and 
    $\lim_{t \to 0} (g \circ \alpha) (t) =0$.  
    This implies that $\lim_{t \to 0} (f \circ \alpha) (t) = 0$ and 
    hence $\alpha \in \Arc{\Z}(f)$. 

    Suppose, now that $\Arc{\Z}(g) \subseteq \Arc{\Z}(f)$. 
    By Theorem \ref{thm:lineqarc2}, there exists $N \in \N$ such that, 
    $h = f^N/g \in \Rrat{R^2}$, which implies that $f^N = gh$ 
    and $f \in \sqrt{\Ang{ g }}$.  
\end{proof}
\begin{corollary}
    \label{cor:nullsz1}
    If $I \subseteq \Rrat{R^2}$ is a finitely generated ideal then 
    $f \in \sqrt{I}$ if and only if $\Arc{\Z}(f) \supseteq \Arc{\Z}(I)$. 
\end{corollary}
\begin{proof}
    By the proof of Theorem \ref{thm:nullsz1}, if $f \in \sqrt{I}$ then 
    there exists $g \in I$ such that $\Arc{\Z}(g) = \Arc{\Z}(I)$, and hence 
    $f \in \sqrt{\Ang{g}}$, and 
    $\Arc{\Z}(f) \supseteq \Arc{\Z}(g)$, which implies, 
    $\Arc{\Z}(f) \supseteq \Arc{\Z}(I)$.  

    Now suppose $\Arc{\Z}(f) \supseteq \Arc{\Z}(I)$. By Lemma 
    \ref{lem:finiteideal1}, 
    there exists $h \in I$ such that $\Arc{\Z}(h) = \Arc{\Z}(I)$. Further, 
    by Proposition \ref{prop:radicals}, 
    $f \in \sqrt{\Ang{h}} \subseteq \sqrt{\Ang{I}}$.
\end{proof}
Similar to the case for regulous functions, every 
finitely generated ideal in $\Rrat{R^2}$ is 
\emph{principally radical} (see \cite{Fic16}). 
\begin{lemma}
    \label{lem:princideal}
    If $I \subseteq \Rrat{R^2}$ is a finitely generated ideal such that 
    $\Arc{\Z}(f) = \Arc{\Z}(I)$ then $\sqrt{\Ang{f}} = \sqrt{I}$. 
\end{lemma}
\begin{proof}
    If $f \in I$ then $\sqrt{\Ang{f}} \subseteq \sqrt{\Ang{I}}$. Now, 
    suppose $g \in \sqrt{I}$, by Corollary \ref{cor:nullsz1}, 
    $\Arc{\Z}(g) \supseteq \Arc{\Z}(I) = \Arc{\Z}(f)$. By Theorem 
    \ref{thm:lineqarc2} (\L{}ojasiewicz inequality), there exists 
    an integer $N$ such that $h \coloneqq g^N/f \in \Rrat{R^2}$, 
    therefore $g^N = fh \in \sqrt{\Ang{f}}$. 
\end{proof}
The following result is a direct consequence of 
Lemmas \ref{lem:princideal} and \ref{lem:finiteideal1}. 
\begin{theorem}
    \label{thm:princidealall}
     If $I \subseteq \Rrat{R^2}$ is a finitely generated ideal, then 
     there exists $f \in \Rrat{R^2}$ such that $\sqrt{\Ang{f}} = \sqrt{I}$. 
\end{theorem}

The following result demonstrates that the extension of a real ideal 
in the ring of polynomials $\Pol{R^2}$ satisfies the 
Nullstellensatz for locally bounded rational functions (Theorem 
\ref{thm:nullsz1}).
\begin{proposition}
    \label{thm:idealpolyclos}
    If $I \subseteq \Pol{R^2}$ is a 
    real ideal then $\sqrt{\Rrat{R^2} \cdot I} = \Arc{\I}(\Arc{\Z}(I))$. 
\end{proposition}
\begin{proof}
    If $f \in \sqrt{\Rrat{R^2} \cdot I}$ then there exists $n \in \N$ such 
    that $f^n = gh$ with $g \in \Rrat{R^2}$ and $h \in I$. By the 
    real Nullstellensatz (\cite[4.46]{BCR13}), since $I$ is real, 
    $h = 0$ on $\Z(I)$. Now, if $\gamma \in \Arc{\Z}(I)$ then 
    $h \circ  \gamma = 0$ which implies that $f^n \circ \gamma = 0$, 
    which, in turn, implies that $\lim_{t \to 0} (f \circ \gamma)(t) =0$, 
    and $f \in \Arc{\I}(\Arc{\Z}(I)$. 

    Now if $g_1, \dots, g_k$ are generators of $I$, let 
    $g = g_1^2 + \dots + g_k^2$. If $f \in \Arc{\I}(\Arc{\Z}(I))$, 
    then by Proposition \ref{prop:idealbasicprop} 
    and Lemma \ref{lem:finiteideal1} $\Arc{\Z}(f) \supseteq \Arc{\Z}(g)$. 
    Now by Theorem \ref{thm:lineqarc2} there exists $n \in N$ such that, 
    $h = f^n/g \in \Rrat{R^2}$. This implies that $f^n = hg \in \Rrat{R^2} \cdot I$
\end{proof}

\bibliographystyle{amsplain}

\bibliography{res_approx}{}

%

\end{document}